\documentclass[12pt]{article}
\usepackage{amsfonts, amsmath, amssymb, latexsym, eucal, amscd} 

\usepackage[dvips]{pict2e}

\usepackage{amsthm}
\usepackage{amscd}

\usepackage[dvipdfmx]{graphics}
\usepackage{cite}

\newtheorem{definition}{Definition}[section]
\newtheorem{theorem}[definition]{Theorem}
\newtheorem{lemma}[definition]{Lemma}
\newtheorem{corollary}[definition]{Corollary}
\newtheorem{remark}[definition]{Remark}
\newtheorem{example}[definition]{Example}
\newtheorem{conjecture}[definition]{Conjecture}
\newtheorem{problem}[definition]{Problem}

\newtheorem{proposition}[definition]{Proposition}

\begin{document} 

\title{\bf The $S_3$-symmetric tridiagonal algebra
}
\author{
Paul Terwilliger}
\date{}
\maketitle
\begin{abstract}
The tridiagonal algebra is defined by two generators and two relations, called the tridiagonal relations. Special cases of the tridiagonal algebra include the $q$-Onsager algebra, the positive part of the $q$-deformed enveloping algebra $U_q({\widehat{\mathfrak{sl}}}_2)$,
and the enveloping algebra of the Onsager Lie algebra.
 In this paper, we introduce the $S_3$-symmetric tridiagonal algebra. This algebra has six generators. The generators can be identified with the vertices of a regular hexagon, such that nonadjacent generators commute and adjacent generators satisfy a pair of tridiagonal relations. For a $Q$-polynomial distance-regular graph $\Gamma$ we turn the tensor power $V^{\otimes 3}$ of the standard module $V$ into a module for an $S_3$-symmetric tridiagonal algebra.
 We investigate in detail the case in which $\Gamma$
 is a Hamming graph. We give some conjectures and open problems.
 \bigskip

 \end{abstract}
 
\section{Introduction}
The algebras discussed in this paper are motivated by a combinatorial object called a $Q$-polynomial distance-regular graph  \cite{bbit, banIto, bcn, dkt, int}.
The first algebra under discussion is called the tridiagonal algebra \cite[Definition~3.9]{qSerre}.
For scalars $\beta, \gamma, \gamma^*, \varrho, \varrho^*$ the tridiagonal algebra $T=T(\beta, \gamma, \gamma^*, \varrho, \varrho^*)$ is defined by two
generators $A, A^*$ and  two relations
\begin{align*}
\lbrack A, A^2 A^* - \beta A A^* A + A^* A^2 - \gamma(A A^*+A^*A) - \varrho A^* \rbrack &=0, \\ 
\lbrack A^*, A^{*2} A - \beta A^* A A^* + A A^{*2} - \gamma^*(A^*A+AA^*) - \varrho^* A \rbrack &=0, 
\end{align*}
\noindent where $\lbrack B,C\rbrack=BC-CB$. The above relations are called the tridiagonal relations  \cite[Section~3]{qSerre}, and they first appeared in \cite[Lemma~5.4]{tSub3}. Special cases of the
tridiagonal algebra include the $q$-Onsager algebra \cite{qOnsager1, qOnsager2, augIto}, the positive part $U^+_q$ of the $q$-deformed enveloping algebra $U_q({\widehat{\mathfrak{sl}}}_2)$  \cite[Example~1.7, Remark~10.2]{someAlg}, 
and the enveloping algebra of the Onsager Lie algebra \cite[Example~3.2, Remark~3.8]{qSerre}. Some notable papers about the tridiagonal algebra are \cite{someAlg, augIto, TDqRacah, 2LT, madrid}.
\medskip

\noindent  We now explain how the tridiagonal algebra appears in the theory of distance-regular graphs. Let $\Gamma$
denote a $Q$-polynomial distance-regular graph \cite[Definition~11.2]{int}. There are some well-known scalars $\beta, \gamma, \gamma^*, \varrho, \varrho^*$ that are used to describe the
eigenvalue sequence and dual eigenvalue sequence of the $Q$-polynomial structure \cite[Lemma~5.4]{tSub3}. Let $X$ denote the vertex set of $\Gamma$. The standard module $V$ of $\Gamma$ is the vector space with basis $X$. 
According to \cite[Lemma~5.4]{tSub3}, for  $x \in X$  the vector space $V$ becomes a module for $T(\beta, \gamma, \gamma^*, \varrho, \varrho^*)$ on which
 $A$ acts as the adjacency map and  $A^*$ acts as the dual adjacency map with respect to $x$ \cite[Section~11]{int}.
\medskip

\noindent Let $S_3$ denote the symmetric group on the set $\lbrace 1,2,3\rbrace$. In the present paper, we introduce a generalization of the tridiagonal algebra called the
$S_3$-symmetric tridiagonal algebra. This algebra is described as follows.
For scalars $\beta, \gamma, \gamma^*, \varrho, \varrho^*$ the $S_3$-symmetric tridiagonal algebra $\mathbb T= \mathbb T(\beta, \gamma, \gamma^*, \varrho, \varrho^* )$
is defined by generators 
\begin{align*}
A_i, \quad A^*_i \qquad \quad i \in \lbrace 1,2,3\rbrace
\end{align*}
and the following relations.
\begin{enumerate}
\item[\rm (i)] For $i,j \in \lbrace 1,2,3\rbrace$,
\begin{align*}
\lbrack A_i, A_j \rbrack=0, \qquad \qquad \lbrack A^*_i, A^*_j \rbrack=0.
\end{align*}
\item[\rm (ii)] For $i \in \lbrace 1,2,3\rbrace$,
\begin{align*}
\lbrack A_i, A^*_i \rbrack=0.
\end{align*}
\item[\rm (iii)] For distinct $i,j \in \lbrace 1,2,3 \rbrace$,
\begin{align*}
\lbrack A_i, A_i^2 A_j^* - \beta A_i A_j^* A_i + A_j^* A_i^2 - \gamma(A_i A_j^*+A_j^*A_i) - \varrho A_j^* \rbrack &=0,  \\
\lbrack A_j^*, A_j^{*2} A_i - \beta A_j^* A_i A_j^* + A_i A_j^{*2} - \gamma^*(A_j^*A_i+A_iA_j^*) - \varrho^* A_i \rbrack &=0. 
\end{align*}
\end{enumerate}

\noindent  Our  presentation of $\mathbb T$  is described by the diagram below:
\vspace{.5cm}

\begin{center}
\begin{picture}(0,60)
\put(-100,0){\line(1,1.73){30}}
\put(-100,0){\line(1,-1.73){30}}
\put(-70,52){\line(1,0){60}}
\put(-70,-52){\line(1,0){60}}
\put(-10,-52){\line(1,1.73){30}}
\put(-10,52){\line(1,-1.73){30}}

\put(-102,-3){$\bullet$}
\put(-72.5,49){$\bullet$}
\put(-72.5,-55){$\bullet$}

\put(16.5,-3){$\bullet$}
\put(-13.5,49){$\bullet$}
\put(-13.5,-55){$\bullet$}

\put(-120,-3){$A_1$}
\put(-11,60){$A_2$}
\put(-11,-66){$A_3$}

\put(26,-3){$A^*_1$}
\put(-82,-66){$A^*_2$}
\put(-82,60){$A^*_3$}

\end{picture}
\end{center}
\vspace{2.5cm}

\hspace{2.2cm}
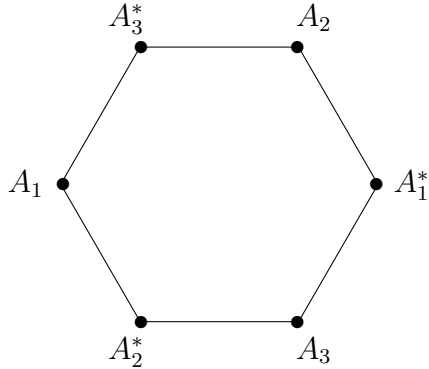
 Fig. 1.  Nonadjacent generators commute. \\
${} \qquad \quad \quad \;\;$ Adjacent generators satisfy the tridiagonal relations.
\vspace{.5cm}

\noindent  We now summarize our main results. Recall the $Q$-polynomial distance-regular graph $\Gamma$ and related scalars $\beta, \gamma, \gamma^*, \varrho, \varrho^*$.
Recall  the standard module $V$,
and consider the tensor power $V^{\otimes 3} = V \otimes V \otimes V$.  We will turn $V^{\otimes 3}$ into a module for $\mathbb T= \mathbb T(\beta, \gamma, \gamma^*, \varrho, \varrho^*)$. 
We will show that the $\mathbb T$-module $V^{\otimes 3}$ has a unique irreducible $\mathbb T$-submodule $\Lambda$ that contains the vector $\sum_{x,y,z \in X} x \otimes y \otimes z$. The $\mathbb T$-submodule
$\Lambda$ is called fundamental. We display some vectors $P_{h,i,j}\in \Lambda$ that are common eigenvectors for $A^*_1, A^*_2, A^*_3$. 
 We display some vectors $Q_{h,i,j} \in \Lambda$ that are common eigenvectors for $A_1, A_2, A_3$. For a subgroup $G$ of the automorphism group of $\Gamma$,
 we show that the natural action of $G$ on $V^{\otimes 3}$ commutes with the 
  $\mathbb T$-action on $V^{\otimes 3}$. We show that $\Lambda$ is contained in the $\mathbb T$-submodule of $V^{\otimes 3}$ consisting of the
 vectors that are fixed by everything in $G$. We consider the case in which $\Gamma$ is a Hamming graph $H(D,N)$ with $D\geq 1$ and $N\geq 3$. For this case, we give an explicit basis for $\Lambda$,
 and the action of the $\mathbb T$-generators on this basis. We also show that ${\rm dim} \,\Lambda = \binom{D+4}{4}$. The graph $H(1,N)$ is the complete graph $K_N$, and we describe
 this special case in detail. We finish with some conjectures and open problems.
\medskip

\noindent We would like  to acknowledge the earlier works about $Q$-polynomial distance-regular graphs that feature a 
tensor power of the standard module $V$. As far as we know, the earliest such work is the 1978 article \cite{norton} by
Cameron, Goethals, and Seidel. In that article $V^{\otimes 2}$ appears in Proposition 5.1 and Remarks 5.3, 5.4, while $V^{\otimes 3}$ appears in the proofs of Propositions 4.1,  5.1.
Much  of \cite{norton} was summarized and popularized in the 1985 book by Bannai and Ito \cite[Section~2.8]{banIto}.
In the 1987 article \cite{balanced} by the present author,  $V^{\otimes 3}$ appears in Section 2 in connection with the balanced-set condition. 
 In the 1995 article \cite{Jaeger2} by Jaeger,  $V^{\otimes 3}$ appears in Section 5 in connection with spin models and the star-triangle relation.
 In the 1995 Ph.D. thesis of Dickie \cite{dickieThesis}, Chapter 4 contains a number of calculations that are implicitly about $V^{\otimes 3}$ and $V^{\otimes 4}$.
 The 1998 works of Suzuki \cite{suzuki1, suzuki2} contain numerous calculations that are implicitly about   $V^{\otimes 3}$ and $V^{\otimes 4}$. 
In their 2003 article \cite{ada2},  Chan,  Godsil, and Munemasa implicitly use $V^{\otimes 2}$ to investigate Jones Pairs.
In his 2021 article \cite{scaffold} about scaffolds, Bill Martin develops a comprehensive diagrammatic approach to computations involving $V^{\otimes n}$ for arbitrary $n$.
As Martin explains in the article, the approach is based on unpublished work of Arnold Neumaier going back to 1989 or before. 
In their 2022 article \cite{safet},
Neumaier and Penji{\' c} develop the diagrammatic approach  using a somewhat different point of view.
The present author acknowledges that he first learned about
the diagrammatic approach from conversations with Neumaier around 1989.
\medskip

\noindent The present paper is organized as follows. Section 2 contains some preliminaries.
In Section 3 we review the tridiagonal algebra $T$.
In Section 4 we introduce the $S_3$-symmetric tridiagonal algebra $\mathbb T$, and establish some basic facts about it.
In Sections 5--8, we use a $Q$-polynomial distance-regular graph $\Gamma$ to turn the tensor power $V^{\otimes 3}$ of the standard module $V$ into a $\mathbb T$-module.
In Section 9 we introduce the fundamental $\mathbb T$-submodule $\Lambda$ of $V^{\otimes 3}$.
In Section 10 we give a group action that commutes with the $\mathbb T$-action on $V^{\otimes 3}$.
In Section 11 we describe $\Lambda$ under the assumption that  $\Gamma$ is a Hamming graph.
In Section 12 we give some conjectures and open problems.

\section{Preliminaries} 
In this section, we review some notation and basic concepts.
Recall the natural numbers $\mathbb N = \lbrace 0,1,2,\ldots \rbrace$.
Let $\mathbb F$ denote a field.
Every vector space and tensor product discussed, is understood to be over $\mathbb F$.
Every algebra without the Lie prefix discussed, is understood to be associative, over $\mathbb F$, and have a multiplicative identity.
A subalgebra has the same multiplicative identity as the parent algebra.
Let $V$ denote a nonzero vector space. The algebra ${\rm End}(V)$  consists of the $\mathbb F$-linear maps
from $V$ to $V$.
An element $A \in {\rm End}(V)$ is said to be {\it diagonalizable} whenever
$V$ is spanned by the eigenspaces of $A$. Assume that $A$ is diagonalizable, and let $\lbrace V_i \rbrace_{i=0}^d$ 
denote an ordering of the eigenspaces of $A$. The sum $V=\sum_{i=0}^d V_i$ is direct.
For $0 \leq i \leq d$ let $\theta_i$ denote the eigenvalue of $A$ for $V_i$.
For $0 \leq i \leq d$ define $E_i \in {\rm End}(V)$ such that
$(E_i-I) V_i=0$ and $E_iV_j=0$ if $j \not=i$ $(0 \leq j \leq d)$. Thus $E_i$ is the projection from $V$ onto $V_i$. We call $E_i$ the {\it primitive idempotent} of $A$ associated with $V_i$ (or $\theta_i$).
By linear algebra 
(i) $A = \sum_{i=0}^d \theta_i E_i$; 
(ii) $E_i E_j = \delta_{i,j} E_i$ $ (0 \leq i,j\leq d)$;
(iii) $I = \sum_{i=0}^d E_i$; 
(iv) $V_i = E_iV$ $ (0 \leq i \leq d)$; 
(v) $AE_i = \theta_i E_i = E_iA$ $(0 \leq i \leq d)$. Moreover
\begin{align*}
  E_i=\prod_{\stackrel{0 \leq j \leq d}{j \neq i}}
          \frac{A-\theta_jI}{\theta_i-\theta_j} \qquad \qquad (0 \leq i \leq d).
\end{align*}

\noindent
Let $0 \not= q \in \mathbb F$.
For elements $B, C$ in any algebra, define
\begin{align*}
\lbrack B,C\rbrack=BC-CB, \qquad \qquad \lbrack B,C \rbrack_q = q BC-q^{-1}CB.
\end{align*}
The symmetric group $S_3$ consists of the permutations of the set $\lbrace 1,2,3\rbrace$.
For matrix representations we use the conventions of  \cite[Section~2]{LSnotes}.

\section{The tridiagonal algebra}
In this section, we recall the tridiagonal algebra \cite{qSerre}.

\begin{definition} \label{def:T} \rm (See \cite[Definition~3.9]{qSerre}.)  
 For $\beta, \gamma, \gamma^*, \varrho, \varrho^* \in \mathbb F$ the algebra $T=T(\beta, \gamma, \gamma^*, \varrho, \varrho^* )$
is defined by generators $A, A^*$ and relations
\begin{align}
\lbrack A, A^2 A^* - \beta A A^* A + A^* A^2 - \gamma(A A^*+A^*A) - \varrho A^* \rbrack &=0, \label{eq:TD1}  \\
\lbrack A^*, A^{*2} A - \beta A^* A A^* + A A^{*2} - \gamma^*(A^*A+AA^*) - \varrho^* A \rbrack &=0. \label{eq:TD2}
\end{align}
We call $T$ the {\it tridiagonal algebra}. The relations \eqref{eq:TD1}, \eqref{eq:TD2} are called the {\it tridiagonal relations}.
\end{definition}

\begin{remark}\rm As far as we know, the relations \eqref{eq:TD1}, \eqref{eq:TD2} first appeared in \cite[Lemma~5.4]{tSub3}.
\end{remark}

\noindent We mention some special cases of the tridiagonal algebra.

\begin{lemma} {\rm (See \cite[Example~3.2, Remark~3.8]{qSerre}.) }
Assume that $\mathbb F$ has characteristic 0. For
\begin{align*}
\beta=2, \qquad \quad \gamma= \gamma^*=0, \qquad \quad \varrho \not=0, \qquad\quad  \varrho^* \not=0
\end{align*}
the tridiagonal relations become the Dolan/Grady relations
\begin{align*}
\lbrack A, \lbrack A, \lbrack A, A^* \rbrack \rbrack \rbrack = \varrho  \lbrack A ,A^*\rbrack, \qquad \quad
\lbrack A^*, \lbrack A^*, \lbrack A^*, A \rbrack \rbrack \rbrack =\varrho^* \lbrack A^*,A\rbrack.
\end{align*}
In this case,
 $T$ becomes the  enveloping algebra $U(O)$ for the  Onsager Lie algebra $O$.
\end{lemma}

\begin{lemma}  {\rm (See \cite[Example~1.7, Remark~10.2]{someAlg}.) }
For $\beta \not=\pm 2$,
\begin{align*}
\beta=q^2 + q^{-2}, \qquad \quad \gamma= \gamma^*=0, \qquad \quad \varrho= \varrho^*=0
\end{align*}
the tridiagonal relations become the $q$-Serre relations
\begin{align*}
\lbrack A, \lbrack A, \lbrack A, A^* \rbrack_q \rbrack_{q^{-1}} \rbrack =0, \qquad \quad 
\lbrack A^*, \lbrack A^*, \lbrack A^*, A \rbrack_q \rbrack_{q^{-1}} \rbrack =0.
\end{align*}
In this case, $T$ becomes the positive part  $U^+_q$ of the $q$-deformed enveloping algebra $U_q( {\widehat{\mathfrak{sl}}}_2)$.
\end{lemma}

\begin{lemma} {\rm (See \cite[Section~2]{qOnsager1}, \cite[Section~1]{qOnsager2}, \cite[Section~1.2]{augIto}.)}
For $\beta \not=\pm 2$,
\begin{align*}
\beta=q^2 + q^{-2}, \qquad \quad \gamma=\gamma^*=0, \qquad \quad \varrho= \varrho^*=-(q^2-q^{-2})^2
\end{align*}
the tridiagonal relations become the $q$-Dolan/Grady relations
\begin{align*}
\lbrack A, \lbrack A, \lbrack A, A^* \rbrack_q \rbrack_{q^{-1}} \rbrack &= (q^2-q^{-2})^2 \lbrack A^*,A\rbrack, \\
\lbrack A^*, \lbrack A^*, \lbrack A^*, A \rbrack_q \rbrack_{q^{-1}} \rbrack &=(q^2-q^{-2})^2 \lbrack A,A^*\rbrack.
\end{align*}
In this case, $T$ becomes the $q$-Onsager algebra $O_q$.
\end{lemma}

\section{The $S_3$-symmetric tridiagonal algebra}

\noindent In this section,  we introduce the $S_3$-symmetric tridiagonal algebra, and establish some basic facts about it.

\begin{definition}\label{def:symT} \rm For $\beta, \gamma, \gamma^*, \varrho, \varrho^* \in \mathbb F$ the algebra $\mathbb T= \mathbb T(\beta, \gamma, \gamma^*, \varrho, \varrho^* )$
is defined by generators 
\begin{align*}
A_i, \quad A^*_i \qquad \quad i \in \lbrace 1,2,3\rbrace
\end{align*}
and the following relations.
\begin{enumerate}
\item[\rm (i)] For $i,j \in \lbrace 1,2,3\rbrace$,
\begin{align*}
\lbrack A_i, A_j \rbrack=0, \qquad \qquad \lbrack A^*_i, A^*_j \rbrack=0.
\end{align*}
\item[\rm (ii)] For $i \in \lbrace 1,2,3\rbrace$,
\begin{align*}
\lbrack A_i, A^*_i \rbrack=0.
\end{align*}
\item[\rm (iii)] For distinct $i,j \in \lbrace 1,2,3 \rbrace$,
\begin{align*}
\lbrack A_i, A_i^2 A_j^* - \beta A_i A_j^* A_i + A_j^* A_i^2 - \gamma(A_i A_j^*+A_j^*A_i) - \varrho A_j^* \rbrack &=0,  \\
\lbrack A_j^*, A_j^{*2} A_i - \beta A_j^* A_i A_j^* + A_i A_j^{*2} - \gamma^*(A_j^*A_i+A_iA_j^*) - \varrho^* A_i \rbrack &=0. 
\end{align*}
\end{enumerate}
We call $\mathbb T$ the {\it $S_3$-symmetric tridiagonal algebra}. 
\end{definition}

\noindent Next, we compare the tridiagonal algebra $T(\beta, \gamma, \gamma^*, \varrho, \varrho^*) $ from Definition \ref{def:T} to the
$S_3$-symmetric tridiagonal algebra  $\mathbb T(\beta, \gamma, \gamma^*, \varrho, \varrho^*)$ from Definition \ref{def:symT}.
 
  \begin{lemma} \label{lem:TtoSymT}
 Referring to Definitions \ref{def:T} and \ref{def:symT},
 for distinct $r,s \in \lbrace 1,2,3\rbrace$ there exists an algebra homomorphism $T(\beta, \gamma, \gamma^*, \varrho, \varrho^*) \to \mathbb T(\beta, \gamma, \gamma^*, \varrho, \varrho^*) $
 that sends 
 \begin{align*}
 A\mapsto A_r, \qquad \qquad A^* \mapsto A^*_s.
 \end{align*}
  \end{lemma}
 \begin{proof} Compare Definitions \ref{def:T} and \ref{def:symT}.
 \end{proof}
\noindent In a moment, we will show that the homomorphism in Lemma \ref{lem:TtoSymT} is injective.
 
 \begin{lemma} \label{lem:comp1} Referring to Definitions \ref{def:T} and \ref{def:symT}, for distinct $r,s \in \lbrace 1,2,3\rbrace$
  there exists an algebra homomorphism $\mathbb T(\beta, \gamma, \gamma^*, \varrho, \varrho^*) \to T(\beta, \gamma, \gamma^*, \varrho, \varrho^*) $
 that sends 
  \begin{align*}
A_i \mapsto
 \begin{cases} A & {\mbox{\rm if $i=r$}}; \\
0, & {\mbox{\rm if $i\not= r$}}
\end{cases} 
\qquad \quad 
A^*_i \mapsto
 \begin{cases} A^* & {\mbox{\rm if $i=s$}}; \\
0, & {\mbox{\rm if $i\not= s$}}
\end{cases} 
\qquad \qquad i \in \lbrace 1,2,3\rbrace.
 \end{align*}
 This homomorphism is surjective.
 \end{lemma}
 \begin{proof} To show that the homomorphism exists, compare Definitions \ref{def:T} and \ref{def:symT}. The last assertion is clear.
 \end{proof}
 
 \begin{lemma} The homomorphism in Lemma \ref{lem:TtoSymT} is injective.
 \end{lemma}
 \begin{proof}  Consider the composition of the homomorphisms in Lemma \ref{lem:TtoSymT} and Lemma \ref{lem:comp1}:
 \begin{align*}
 T(\beta, \gamma, \gamma^*, \varrho, \varrho^*) \to \mathbb T(\beta, \gamma, \gamma^*, \varrho, \varrho^*) \to T(\beta, \gamma, \gamma^*, \varrho, \varrho^*).
 \end{align*}
 This composition sends
 \begin{align*}
 A \mapsto A_r \mapsto A, \qquad \qquad A^* \mapsto A^*_s \mapsto A^*
 \end{align*}
 and is therefore the identity map on  $T(\beta, \gamma, \gamma^*, \varrho, \varrho^*)$. The result follows.
 \end{proof}

\noindent For the rest of this section, we describe the basic structure and symmetries of the algebra $\mathbb T=\mathbb T(\beta, \gamma, \gamma^*, \varrho, \varrho^*)$ from Definition \ref{def:symT}.
\begin{definition} \label{def:poly} \rm Consider the following mutually commuting indeterminates:
\begin{align} \label{eq:six}
\widehat {A_1}, \quad \widehat{A_2}, \quad  \widehat {A_3}, \quad \widehat {A^*_1}, \quad \widehat {A^*_2}, \quad  \widehat {A^*_3}.
\end{align}
Let $\mathbb F\lbrack \widehat {A_1}, \widehat {A_2},   \widehat {A_3},  \widehat {A^*_1},  \widehat {A^*_2},  \widehat {A^*_3} \rbrack$ denote the
algebra consisting of the polynomials in the indeterminates  \eqref{eq:six} that have all coefficients in $\mathbb F$.
\end{definition}

\begin{lemma}\label{lem:abel} There exists an algebra homomorphism $ \natural: \mathbb T \to  \mathbb F\lbrack \widehat {A_1}, \widehat {A_2},   \widehat {A_3},  \widehat {A^*_1},  \widehat {A^*_2},  \widehat {A^*_3} \rbrack$ 
that sends $A_i \mapsto \widehat {A_i}$ and  $A^*_i \mapsto \widehat {A^*_i}$ for $i \in \lbrace 1,2,3\rbrace$. This homomorphism is surjective.
\end{lemma}
\begin{proof} The generators \eqref{eq:six} mutually commute, so they satisfy the defining relations for $\mathbb T$ given in Definition \ref{def:symT}.
Therefore, the algebra homomorphism exists. It is clear that the homomorphism is surjective.
\end{proof}

\begin{lemma} \label{lem:lind} The following elements are linearly independent in $\mathbb T$:
\begin{align} \label{eq:lind1}
A_1^hA_2^i A_3^j {A^*_1}^r  {A^*_2}^s  {A^*_3}^t  \qquad \qquad h,i,j,r,s,t \in \mathbb N.
\end{align}
Moreover, the following elements are linearly independent in $\mathbb T$:
\begin{align} \label{eq:lind2}
{A^*_1}^h {A^*_2}^i  {A^*_3}^j A_1^rA_2^s A_3^t  \qquad \qquad h,i,j,r,s,t \in \mathbb N.
\end{align}
\end{lemma}
\begin{proof} The elements \eqref{eq:lind1} are linearly independent in $\mathbb T$, because their $\natural$-images are linearly 
independent. The elements \eqref{eq:lind2} are linearly independent in $\mathbb T$, for the same reason.
\end{proof}

\begin{lemma} \label{lem:4} The following {\rm (i)}--{\rm (iv)} hold in $\mathbb T$:
\begin{enumerate}
\item[\rm (i)] $A_1, A_2, A_3$ are algebraically independent;
\item[\rm (ii)] $A^*_1, A^*_2, A^*_3$ are algebraically independent;
\item[\rm (iii)] $A_i, A^*_i$ are algebraically independent for $i \in \lbrace 1,2,3\rbrace$;
\item[\rm (iv)] the following seven elements are linearly independent:
\begin{align*}
1, \quad A_1, \quad A_2, \quad A_3, \quad A^*_1, \quad A^*_2, \quad A^*_3.
\end{align*}
\end{enumerate}
\end{lemma}
\begin{proof} Immediate from Lemma \ref{lem:lind}.
\end{proof}

\noindent By an {\it automorphism of $\mathbb T$}, we mean an algebra isomorphism $\mathbb T \to \mathbb T$.
The automorphism group  ${\rm Aut}(\mathbb T)$ consists of the automorphisms of $\mathbb T$; the group operation is composition.

\begin{lemma} \label{lem:aut} The following {\rm (i)--\rm (iii)} hold.
\begin{enumerate}
\item[\rm (i)]
For  $\sigma \in S_3$ there exists $\hat \sigma \in {\rm Aut}(\mathbb T)$ that sends
\begin{align*}
A_i \mapsto A_{\sigma(i)}, \qquad \quad A^*_i \mapsto A^*_{\sigma(i)} \qquad \quad i \in \lbrace 1,2,3 \rbrace.
\end{align*}
\item[\rm (ii)] 
The map $S_3 \to {\rm Aut}(\mathbb T)$, $ \sigma \mapsto {\hat \sigma}$ is a homomorphism of groups.
\item[\rm (iii)] 
The homomorphism in {\rm (ii)} is injective.
\end{enumerate}
\end{lemma}
\begin{proof} (i) By the nature of the defining relations in Definition \ref{def:symT}. \\
\noindent (ii) This is readily checked. \\
\noindent (iii) By 
 Lemma \ref{lem:4}(iv).
 \end{proof}
 
 \begin{lemma} \label{lem:dual} Referring to Definition \ref{def:symT}, there exists an algebra isomorphism
 \begin{align*}
 \mathbb T(\beta, \gamma, \gamma^*, \varrho, \varrho^*) \to \mathbb T(\beta, \gamma^*, \gamma, \varrho^*, \varrho)
 \end{align*}
 that sends 
 \begin{align*}
 A_i \mapsto A^*_i, \qquad A^*_i \mapsto A_i \qquad \quad  i \in \lbrace 1,2,3\rbrace.
 \end{align*}
 \end{lemma}
 \begin{proof} Compare the defining relations for $ \mathbb T(\beta, \gamma, \gamma^*, \varrho, \varrho^*) $ and $\mathbb T(\beta, \gamma^*, \gamma, \varrho^*, \varrho)$.
 \end{proof}

\section{A module for the $S_3$-symmetric tridiagonal algebra}

In this section, we use a $Q$-polynomial distance-regular graph to construct a module for the $S_3$-symmetric tridiagonal algebra. For the basic
facts about this type of graph, we refer the reader to \cite{bbit, banIto, bcn, dkt, int}. In what follows, we will generally adopt the point of view from \cite{int}.
\medskip

\noindent From now until the end of Section 11, the following assumptions and notation are in effect. Let $\mathbb R$ (resp. $\mathbb C$) denote the field of real numbers (resp. complex numbers).
Let $\mathbb F=\mathbb C$. Let $\Gamma=(X,\mathcal R)$ denote a distance-regular graph \cite[Section~2]{int} with vertex set $X$,
adjacency relation $\mathcal R$, path-length distance function $\partial$, and  diameter $D\geq 1$. For $x \in X$ and $0 \leq i \leq D$ define the set $\Gamma_i(x)=\lbrace y \in X \vert \partial(x,y)=i\rbrace$. We abbreviate $\Gamma(x)=\Gamma_1(x)$.
Assume that $\Gamma$ is $Q$-polynomial \cite[Definition~11.1]{int}, with eigenvalue sequence $\lbrace \theta_i \rbrace_{i=0}^D$ \cite[Definition~3.6]{int} and  dual eigenvalue sequence $\lbrace \theta^*_i \rbrace_{i=0}^D$ \cite[Definition~11.8]{int}. 
By construction $\theta_i, \theta^*_i \in \mathbb R$ for $ 0 \leq i \leq D$. By \cite[Lemma~3.5]{int} the scalars $\lbrace \theta_i \rbrace_{i=0}^D$ are mutually distinct.
By \cite[Lemma~11.7]{int} the scalars  $\lbrace \theta^*_i \rbrace_{i=0}^D$ are mutually distinct.
\medskip

\noindent The following result is well known; see for example \cite[Lemma~5.4]{tSub3} or \cite[Proposition~15.9]{int}.
\begin{lemma} \label{lem:TTR} {\rm (See \cite[Lemma~5.4]{tSub3}.)} There exist scalars  $\beta, \gamma, \gamma^*, \varrho, \varrho^* \in \mathbb R$ that
satisfy the following {\rm (i)--(iii)}.
\begin{enumerate}
\item[\rm (i)]  $\beta+1$ is equal to each of
\begin{align*}
\frac{\theta_{i-2}-\theta_{i+1}}{\theta_{i-1}-\theta_i}, \qquad \qquad  \frac{\theta^*_{i-2}-\theta^*_{i+1}}{\theta^*_{i-1}-\theta^*_i}
\end{align*}
for $2 \leq i \leq D-1$.
\item[\rm (ii)] For $1 \leq i \leq D-1$, both
\begin{align*}
\gamma = \theta_{i-1}- \beta \theta_i + \theta_{i+1}, \qquad \qquad
\gamma^* = \theta^*_{i-1}- \beta \theta^*_i + \theta^*_{i+1}.
\end{align*}
\item[\rm (iii)] For $1 \leq i \leq D$, both
\begin{align*}
\varrho &= \theta^2_{i-1} - \beta \theta_{i-1} \theta_i + \theta^2_i - \gamma (\theta_{i-1}+\theta_i), \\
\varrho^* &= \theta^{*2}_{i-1} - \beta \theta^*_{i-1} \theta^*_i + \theta^{*2}_i - \gamma^* (\theta^*_{i-1}+\theta^*_i).
\end{align*}

\end{enumerate}
\end{lemma}

\begin{definition}\rm Let $V$ denote a vector space over $\mathbb C$ with basis $X$. We call $V$
the {\it standard module} associated with $\Gamma$.
\end{definition}
\begin{definition} \rm We define the vector space $V^{\otimes 3} = V \otimes V \otimes V$ and the set
\begin{align*}
X^{\otimes 3} = \lbrace x \otimes y \otimes z\vert x,y,z \in X \rbrace.
\end{align*}
Note that $X^{\otimes 3}$ is a basis for $V^{\otimes 3}$.
\end{definition}
\noindent We now state our first main result.

\begin{theorem}\label{thm:main} For the scalars $\beta, \gamma, \gamma^*, \varrho, \varrho^*$ from Lemma \ref{lem:TTR}, the vector space $V^{\otimes 3}$ becomes a $\mathbb T(\beta, \gamma, \gamma^*, \varrho, \varrho^*)$-module
on which the generators $\lbrace A_i \rbrace_{i=1}^3$,  $\lbrace A^*_i \rbrace_{i=1}^3$ act as follows.
For $x,y,z \in X$,
\begin{align*}
A_1 (x\otimes y \otimes z) &= \sum_{\xi \in \Gamma(x)} \xi \otimes y \otimes z, \\
A_2 (x\otimes y \otimes z) &= \sum_{\xi \in \Gamma(y)} x \otimes \xi \otimes z, \\
A_3 (x\otimes y \otimes z) &= \sum_{\xi \in \Gamma(z)} x \otimes y \otimes \xi, \\
A^*_1 (x \otimes y \otimes z) &= x \otimes y \otimes z \,\theta^*_{\partial(y,z)},  \\
A^*_2 (x \otimes y \otimes z) &= x \otimes y \otimes z \, \theta^*_{\partial(z,x)},  \\
A^*_3 (x \otimes y \otimes z) &= x \otimes y \otimes z \,\theta^*_{\partial(x,y)}.
\end{align*}
\end{theorem}

\noindent The proof of Theorem \ref{thm:main} will be completed in Section 8.

\begin{remark}\rm The six actions shown in Theorem \ref{thm:main} are discussed in \cite[p.~76]{scaffold}. The actions of $A_1$, $A_2$, $A_3$ are called node actions, and 
the actions of $A^*_1$, $A^*_2$, $A^*_3$ are called edge actions.
\end{remark}

\section{The maps  $A^{(1)}, A^{(2)}, A^{(3)}$}
\noindent  We continue to discuss the $Q$-polynomial distance-regular graph  $\Gamma=(X,\mathcal R)$ from Section 5.  
Recall the standard module $V$. 
\begin{definition} \label{def:adj} \rm Define $A \in {\rm End}(V)$ such that 
\begin{align*}
A x = \sum_{\xi \in \Gamma(x)} \xi, \qquad \qquad x \in X.         
\end{align*}
We call $A$ the {\it adjacency map} for $\Gamma$.
\end{definition}

\noindent Recall the vector space $V^{\otimes 3} = V \otimes V \otimes V$.  

\begin{definition}\rm \label{def:mapsAAA}
We define $A^{(1)}, A^{(2)}, A^{(3)} \in {\rm End}(V^{\otimes 3})$ as follows. For $x,y,z \in X$,
\begin{align*}
A^{(1)} (x\otimes y \otimes z) &= \sum_{\xi \in \Gamma(x)} \xi \otimes y \otimes z, \\
A^{(2)} (x\otimes y \otimes z) &= \sum_{\xi \in \Gamma(y)} x \otimes \xi \otimes z, \\
A^{(3)} (x\otimes y \otimes z) &= \sum_{\xi \in \Gamma(z)} x \otimes y \otimes \xi.
\end{align*}
\end{definition}

\begin{lemma} \label{lem:Atensor} For $u,v,w \in V$ we have
\begin{align*}
&A^{(1)} (u\otimes v \otimes w) = Au\otimes v \otimes w, \qquad \quad 
A^{(2)} (u\otimes v \otimes w) = u\otimes Av \otimes w, \\
&A^{(3)} (u\otimes v \otimes w) = u\otimes v \otimes Aw.
\end{align*}
\end{lemma}
\begin{proof} Routine consequence of Definitions \ref{def:adj}, \ref{def:mapsAAA}.
\end{proof}

\noindent By \cite[Section~2 and Lemma~3.5]{int} the map $A$ is diagonalizable on $V$, with eigenvalues $\lbrace \theta_i \rbrace_{i=0}^D$. For $0 \leq i \leq D$ let $E_i$ denote the primitive idempotent of $A$ for $\theta_i$.
Note that $E_iV$ is the $\theta_i$-eigenspace of $A$.

\begin{lemma} \label{lem:Aeig} Each of the maps $A^{(1)}, A^{(2)}, A^{(3)} $ is diagonalizable, with eigenvalues $\lbrace \theta_i \rbrace_{i=0}^D$. For $0 \leq i \leq D$
their $\theta_i$-eigenspaces are
\begin{align*}
E_iV \otimes V \otimes V, \qquad \quad V \otimes E_iV \otimes V, \qquad \qquad V \otimes V \otimes E_iV,
\end{align*}
respectively.
\end{lemma}
\begin{proof} By $S_3$-symmetry, it suffices to prove the result for $A^{(1)}$. By Lemma  \ref{lem:Atensor} and the construction, for $0 \leq i \leq D$ each vector in
$E_iV \otimes V \otimes V$ is an eigenvector for $A^{(1)}$ with eigenvalue $\theta_i$. The sum $V=\sum_{i=0}^D E_iV$ is direct, so the following sum is direct:
\begin{align*}
V^{\otimes 3} = \sum_{i=0}^D E_iV \otimes V \otimes V.
\end{align*}
By the above  comments, we get the result for $A^{(1)}$.
\end{proof}

\noindent Next we describe the primitive idempotents for  $A^{(1)}, A^{(2)}, A^{(3)}$.

\begin{definition}\rm \label{def:PI3} For $0 \leq i \leq D$ let  $E^{(1)}_i$ (resp. $E^{(2)}_i$) (resp. $E^{(3)}_i$)  denote the primitive idempotent of $A^{(1)}$ (resp. $A^{(2)}$) (resp. $A^{(3)}$) for $\theta_i$.
\end{definition}

\begin{lemma} \label{lem:E3action}  For $0 \leq i \leq D$ the maps $E_i^{(1)}, E_i^{(2)}, E_i^{(3)}$ act as follows. For $u,v,w \in V$,
\begin{align*}
&E_i^{(1)} (u\otimes v \otimes w) = E_iu\otimes v \otimes w, \qquad \quad 
E_i^{(2)} (u\otimes v \otimes w) = u\otimes E_iv \otimes w, \\
&E_i^{(3)} (u\otimes v \otimes w) = u\otimes v \otimes E_iw.
\end{align*}
\end{lemma}
\begin{proof} By $S_3$-symmetry, it suffices to prove the result for $E_i^{(1)}$. Note that the map
\begin{align*}
V^{\otimes 3} &\to V^{\otimes 3} \\
 u \otimes v \otimes w &\mapsto E_i u\otimes v \otimes w
\end{align*}
acts as the identity on $E_iV \otimes V\otimes V$, and as 0 on  $E_jV \otimes V\otimes V$ for $0 \leq j \leq D$, $j\not=i$. By these comments and Lemma  \ref{lem:Aeig}, we get the result for $E_i^{(1)}$.
\end{proof}

\begin{lemma} \label{lem:E3prop} For $r \in \lbrace 1,2,3\rbrace $ we have
\begin{align*}
&A^{(r)} = \sum_{i=0}^D \theta_i E_i^{(r)}, \qquad \qquad \quad I = \sum_{i=0}^D E_i^{(r)}, \\
& E_i^{(r)} E_j^{(r)} = \delta_{i,j} E_i^{(r)} \qquad \qquad \quad (0 \leq i,j\leq D), \\
& A^{(r)} E_i^{(r)} = \theta_i E_i^{(r)} = E_i^{(r)} A^{(r)} \qquad  (0 \leq i \leq D), \\
& E_i^{(r)} = \prod_{\stackrel{0 \leq j \leq D}{j \neq i}} \frac{A^{(r)} - \theta_j I}{\theta_i - \theta_j} \qquad \qquad (0 \leq i \leq D).
\end{align*} 
\end{lemma}
\begin{proof} By Definition \ref{def:PI3} and the discussion about primitive idempotents in Section 2.
\end{proof}

\noindent Next, we describe how  $A^{(1)}, A^{(2)}, A^{(3)}$ are related.

\begin{lemma} \label{lem:ABC} The maps  $A^{(1)}, A^{(2)}, A^{(3)}$ mutually commute. Their common eigenspaces are
\begin{align*}
E_hV \otimes E_i V\otimes E_j V \qquad \qquad (0 \leq h,i,j\leq D).
\end{align*}
\end{lemma}
\begin{proof} The following sum is direct:
\begin{align*}
 V^{\otimes 3} = \sum_{h=0}^D \sum_{i=0}^D \sum_{j=0}^D E_h V \otimes E_i V \otimes E_j V.
 \end{align*}
 By Lemma \ref{lem:Aeig}, for $0 \leq h,i,j\leq D$ each element in $E_h V \otimes E_i V \otimes E_j V$ is an eigenvector  for $A^{(1)}$ (resp. $A^{(2)}$) (resp. $ A^{(3)}$) with 
 eigenvalue $\theta_h$ (resp. $\theta_i$) (resp. $\theta_j$). The result follows.
\end{proof}

\noindent We end this section with a comment about  the dual eigenvalue sequence $\lbrace \theta^*_i \rbrace_{i=0}^D$.
\begin{lemma} {\rm (See \cite[Section~19]{int}.)} \label{lem:dpi}
For $z \in X$,
\begin{align*}
E_1 z = \vert X \vert^{-1} \sum_{\xi \in X} \xi \theta^*_{\partial(\xi, z)}.
\end{align*}
\end{lemma}

\section{The maps  $A^{*(1)}, A^{*(2)}, A^{*(3)}$}
\noindent  We continue to discuss the $Q$-polynomial distance-regular graph  $\Gamma=(X,\mathcal R)$ from Section 5.  
Recall the standard module $V$ and the vector space $V^{\otimes 3} = V \otimes V \otimes V$.  

\begin{definition}\rm \label{def:mapsAAAs}
We define $A^{*(1)}, A^{*(2)}, A^{*(3)} \in {\rm End}(V^{\otimes 3})$ as follows. For $x,y,z \in X$,
\begin{align*}
A^{*(1)} (x\otimes y \otimes z) &=  x \otimes y \otimes z\,\theta^*_{\partial(y,z)}, \\
A^{*(2)} (x\otimes y \otimes z) &=  x \otimes y\otimes z\, \theta^*_{\partial(z,x)}  , \\
A^{*(3)} (x\otimes y \otimes z) &=  x \otimes y \otimes z\, \theta^*_{\partial(x,y)}.
\end{align*}
\end{definition}

\begin{lemma} \label{lem:AAAsEig} Each of the maps $A^{*(1)}, A^{*(2)}, A^{*(3)} $ is diagonalizable, with eigenvalues $\lbrace \theta^*_i \rbrace_{i=0}^D$. For $0 \leq i \leq D$
their $\theta^*_i$-eigenspaces are
\begin{align*}
&{\rm Span} \bigl \lbrace x \otimes y \otimes z \vert x,y,z \in X, \partial(y,z)=i\bigr\rbrace, \\
&{\rm Span} \bigl \lbrace x \otimes y \otimes z \vert x,y,z \in X, \partial(z,x)=i\bigr\rbrace, \\
&{\rm Span} \bigl \lbrace x \otimes y \otimes z \vert x,y,z \in X, \partial(x,y)=i\bigr\rbrace.
\end{align*}
respectively.
\end{lemma}
\begin{proof} We invoke Definition \ref{def:mapsAAAs}. By $S_3$-symmetry, it suffices to prove the result for  $A^{*(1)}$. The set  $X^{\otimes 3}$  forms a basis for $V^{\otimes 3}$ consisting
of eigenvectors for  $A^{*(1)}$.
For $x,y,z \in X$ the eigenvector $x \otimes y \otimes z$ has eigenvalue $\theta^*_i$, where $i = \partial(y,z)$. The possible values of $\partial(y,z)$
are $\lbrace 0,1,\ldots, D\rbrace$ so the eigenvalues of  $A^{*(1)}$ are $\lbrace \theta^*_i \rbrace_{i=0}^D$. By these comments, we get the result for  $A^{*(1)}$.

\end{proof}

\noindent Next we describe the primitive idempotents for  $A^{*(1)}, A^{*(2)}, A^{*(3)}$.

\begin{definition}\rm \label{def:PI3s} For $0 \leq i \leq D$ let  $E^{*(1)}_i$ (resp. $E^{*(2)}_i$) (resp. $E^{*(3)}_i$)  denote the primitive idempotent of $A^{*(1)}$ (resp. $A^{*(2)}$) (resp. $A^{*(3)}$) for $\theta^*_i$.
\end{definition}

\begin{lemma} \label{lem:E3saction}  For $0 \leq i \leq D$ the maps $E_i^{*(1)}, E_i^{*(2)}, E_i^{*(3)}$ act as follows. For $x,y,z \in X$,
\begin{align*}
&E_i^{*(1)} (x\otimes y \otimes z) = \begin{cases} x\otimes y \otimes z,& {\mbox{\rm if $\partial(y,z)=i$}}; \\
0, & {\mbox{\rm if $\partial(y,z) \not= i$}}
\end{cases} \\
&E_i^{*(2)} (x\otimes y \otimes z) = \begin{cases} x\otimes y \otimes z,& {\mbox{\rm if $\partial(z,x)=i$}}; \\
0, & {\mbox{\rm if $\partial(z,x) \not= i$}}
\end{cases} \\
&E_i^{*(3)} (x\otimes y \otimes z) = \begin{cases} x\otimes y \otimes z,& {\mbox{\rm if $\partial(x,y)=i$}}; \\
0, & {\mbox{\rm if $\partial(x,y) \not= i$}}.
\end{cases} 
\end{align*}
\end{lemma}
\begin{proof} By $S_3$-symmetry, it suffices to prove the result for  $E_i^{*(1)}$. By Lemma  \ref{lem:AAAsEig}
and Definition  \ref{def:PI3s},
$E_i^{*(1)}$ acts as the identity on 
${\rm Span} \bigl \lbrace x \otimes y \otimes z \vert x,y,z \in X, \partial(y,z)=i\bigr\rbrace$ and as zero on 
${\rm Span} \bigl \lbrace x \otimes y \otimes z \vert x,y,z \in X, \partial(y,z)\not=i\bigr\rbrace$. By these comments we get the result for
 $E_i^{*(1)}$.
\end{proof}

\begin{lemma} \label{lem:E3sprop} For $r \in \lbrace 1,2,3\rbrace $ we have
\begin{align*}
&A^{*(r)} = \sum_{i=0}^D \theta^*_i E_i^{*(r)}, \qquad \qquad \qquad I = \sum_{i=0}^D E_i^{*(r)}, \\
& E_i^{*(r)} E_j^{*(r)} = \delta_{i,j} E_i^{*(r)} \qquad \qquad \qquad (0 \leq i,j\leq D), \\
& A^{*(r)} E_i^{*(r)} = \theta^*_i E_i^{*(r)} = E_i^{*(r)} A^{*(r)} \qquad  (0 \leq i \leq D), \\
& E_i^{*(r)} = \prod_{\stackrel{0 \leq j \leq D}{j \neq i}} \frac{A^{*(r)} - \theta^*_j I}{\theta^*_i - \theta^*_j} \qquad \qquad (0 \leq i \leq D).
\end{align*} 
\end{lemma}
\begin{proof}  By Definition \ref{def:PI3s} and the discussion about primitive idempotents in Section 2.
\end{proof}

\noindent Next, we describe how  $A^{*(1)}, A^{*(2)}, A^{*(3)}$ are related. Recall from \cite[Section~2]{int} the intersection numbers $p^h_{i,j}$ $(0 \leq h,i,j\leq D)$.

\begin{lemma} \label{lem:ABCs} The maps  $A^{*(1)}, A^{*(2)}, A^{*(3)}$ mutually commute. Their common eigenspaces are
\begin{align*}
&{\rm Span} \bigl\lbrace x \otimes y \otimes z \vert x,y,z\in X,  \partial(y,z)=h, \partial(z,x)=i, \partial(x,y)=j \bigr\rbrace,\\
& 0 \leq h,i,j\leq D, \qquad \qquad  p^h_{i,j} \not=0.
\end{align*}
\end{lemma}
\begin{proof} The set $X^{\otimes 3}$  forms a basis for $V^{\otimes 3}$ consisting of
common eigenvectors for  $A^{*(1)}$, $A^{*(2)}$, $A^{*(3)}$. For $0 \leq h,i,j\leq D$ the following are equivalent: (i) $p^h_{i,j} \not=0$;
(ii) there exists $x,y,z \in X$ such that $h= \partial(y,z)$, $i= \partial(z,x)$, $j= \partial(x,y)$. Suppose the equivalent conditions (i), (ii) hold,
and let $x,y,z$ satisfy (ii). 
Then the eigenvector $x \otimes y \otimes z$ has eigenvalue
$\theta^*_h$ (resp. $\theta^*_i$) (resp. $\theta^*_j$) for  $A^{*(1)}$ (resp. $A^{*(2)}$) (resp. $ A^{*(3)}$). By these comments we get the result.
\end{proof}

\section{ How    $A^{(1)}, A^{(2)}, A^{(3)}$ and   $A^{*(1)}, A^{*(2)}, A^{*(3)}$ are related}

\noindent  We continue to discuss the $Q$-polynomial distance-regular graph  $\Gamma=(X,\mathcal R)$ from Section 5.
In this section, we describe how the maps   $A^{(1)}, A^{(2)}, A^{(3)}$ from Definition  \ref{def:mapsAAA} are related to the maps
   $A^{*(1)}, A^{*(2)}, A^{*(3)}$ from Definition  \ref{def:mapsAAAs}.

\begin{proposition} \label{lem:AASB} We have
\begin{align*}
\lbrack A^{(1)}, A^{*(1)} \rbrack=0, \qquad \qquad
\lbrack A^{(2)}, A^{*(2)} \rbrack=0, \qquad \qquad
\lbrack A^{(3)}, A^{*(3)} \rbrack=0.
\end{align*}
\end{proposition} 
\begin{proof} By $S_3$-symmetry, it suffices to prove that $\lbrack A^{(1)}, A^{*(1)} \rbrack=0$. The maps $A^{(1)}$, $A^{*(1)}$ commute  because for $x,y,z \in X$, both
\begin{align*}
&A^{(1)} A^{*(1)} (x \otimes y \otimes z) = A^{(1)}  (x \otimes y \otimes z)  \theta^*_{\partial(y,z)} =
 \theta^*_{\partial(y,z)} \sum_{\xi \in \Gamma(x)} \xi \otimes y \otimes z, \\
& A^{*(1)} A^{(1)} (x \otimes y \otimes z) = A^{*(1)} \sum_{\xi \in \Gamma(x)} \xi \otimes y \otimes z =  \theta^*_{\partial(y,z)} \sum_{\xi \in \Gamma(x)} \xi \otimes y \otimes z.
\end{align*}
\end{proof}

\begin{lemma} \label{prop:EsAEs} For $0 \leq i,j\leq D$ such that $\vert i-j \vert >1$, we have
\begin{align*}
&E_i^{*(2)} A^{(1)} E_j^{*(2)} =0, \qquad \qquad E_i^{*(3)} A^{(1)} E_j^{*(3)} =0, \\
&E_i^{*(3)} A^{(2)} E_j^{*(3)} =0, \qquad \qquad E_i^{*(1)} A^{(2)} E_j^{*(1)} =0, \\
&E_i^{*(1)} A^{(3)} E_j^{*(1)} =0, \qquad \qquad E_i^{*(2)} A^{(3)} E_j^{*(2)} =0.
\end{align*}
\end{lemma}
\begin{proof} By $S_3$-symmetry, it suffices to prove that $E_i^{*(2)} A^{(1)} E_j^{*(2)} =0$. To prove this equation, we show that for $x,y,z\in X$,
\begin{align} \label{eq:EAExyz}
E_i^{*(2)} A^{(1)} E_j^{*(2)}(x \otimes y \otimes z) = 0.
\end{align}
First assume that $\partial(x,z) \not=j$. Then \eqref{eq:EAExyz} holds because  $E_j^{*(2)}(x \otimes y \otimes z) = 0$. Next assume that $\partial(x,z)=j$.
Then
\begin{align*}
&E_i^{*(2)} A^{(1)} E_j^{*(2)}(x \otimes y \otimes z) = E_i^{*(2)} A^{(1)} (x \otimes y \otimes z) \\
& \qquad = E_i^{*(2)} \sum_{\xi \in \Gamma(x)}  \xi \otimes y \otimes z 
=  \sum_{\xi \in \Gamma(x) \cap \Gamma_i(z)}  \xi \otimes y \otimes z =0,
\end{align*}
\noindent  with the last equality holding because the set $ \Gamma(x) \cap \Gamma_i(z)$ is empty by the triangle inequality and $\vert i-j\vert >1$. We have shown \eqref{eq:EAExyz}.
By these comments, we get $E_i^{*(2)} A^{(1)} E_j^{*(2)} =0$.
\end{proof}

\begin{proposition} \label{prop:AAsA}
For distinct $r,s \in \lbrace 1,2,3 \rbrace$ we have
\begin{align*}
\bigl \lbrack A^{*(s)}, A^{*(s)2} A^{(r)}  - \beta A^{*(s)}  A^{(r)} A^{*(s)} + A^{(r)} A^{*(s)2} - \gamma^* (A^{*(s)}  A^{(r)} +A^{(r)}A^{*(s)}) - \varrho^* A^{(r)} \bigr \rbrack &=0.
\end{align*}
\end{proposition}
\begin{proof} Let $C$ denote the expression on the left. We show that $C=0$. We have
\begin{align*}
C = I C I = \Biggl( \sum_{i=0}^D E_i^{*(s)} \Biggr) C  \Biggl( \sum_{j=0}^D E_j^{*(s)} \Biggr) = \sum_{i=0}^D \sum_{j=0}^D E_i^{*(s)}C E_j^{*(s)}.
\end{align*}
For $0\leq i,j\leq D$ we show that $E_i^{*(s)}C E_j^{*(s)}=0$. Using $E_i^{*(s)} A^{*(s)} = \theta^*_i E_i^{*(s)}$ and  $A^{*(s)}E_j^{*(s)}  = \theta^*_j E_j^{*(s)}$, we obtain
\begin{align} \label{eq:terms}
E_i^{*(s)}C E_j^{*(s)}=E_i^{*(s)} A^{(r)}  E_j^{*(s)} (\theta^*_i - \theta^*_j)P^*(\theta^*_i, \theta^*_j),
\end{align}
where the polynomial $P^*(\lambda, \mu)$ is defined by
\begin{align*}
P^*(\lambda, \mu) = \lambda^2 - \beta \lambda \mu + \mu^2 - \gamma^*(\lambda + \mu) - \varrho^*.
\end{align*}
We examine the factors on the right in \eqref{eq:terms}.
If  $\vert i - j \vert >1$ then $E_i^{*(s)} A^{(r)}  E_j^{*(s)}=0$ by Lemma  \ref{prop:EsAEs}.
If $\vert i-j\vert = 1$ then $P^*(\theta^*_i, \theta^*_j)=0$ by Lemma  \ref{lem:TTR}(iii).
If $i=j$ then of course  $\theta^*_i - \theta^*_j=0$.
By these comments, the expression on the right in  \eqref{eq:terms} is equal to zero.
We have shown that $E_i^{*(s)}C E_j^{*(s)}=0$ for $0 \leq i,j\leq D$. Therefore $C=0$.
\end{proof}

\noindent We bring in some notation.
For $u,v\in V$ we define a vector $u \circ v \in V$ as follows. Write 
\begin{align*}
u=\sum_{x \in X} u_x x, \qquad \qquad  v = \sum_{x \in X} v_x x, \qquad \qquad u_x, v_x \in \mathbb C.
\end{align*}
We define
\begin{align}
u \circ v = \sum_{x \in X} u_x v_x x. \label{eq:circMeaning}
\end{align}
\begin{lemma} {\rm (See 
\cite[Theorem~9.4 and Definition~11.1]{int}.)}
 \label{lem:circ} For $0 \leq i,j\leq D$ such that $\vert i-j\vert >1$,
\begin{align*}
E_i(E_1V \circ E_jV) = 0.
\end{align*}
\end{lemma}

\begin{lemma} \label{prop:EAsE} For $0 \leq i,j\leq D$ such that $\vert i-j \vert >1$, we have
\begin{align*}
&E_i^{(2)} A^{*(1)} E_j^{(2)} =0, \qquad \qquad E_i^{(3)} A^{*(1)} E_j^{(3)} =0, \\
&E_i^{(3)} A^{*(2)} E_j^{(3)} =0, \qquad \qquad E_i^{(1)} A^{*(2)} E_j^{(1)} =0, \\
&E_i^{(1)} A^{*(3)} E_j^{(1)} =0, \qquad \qquad E_i^{(2)} A^{*(3)} E_j^{(2)} =0.
\end{align*}
\end{lemma}
\begin{proof} By $S_3$-symmetry, it suffices to prove that $E_i^{(2)} A^{*(1)} E_j^{(2)} =0$.  To prove this equation,
we show that for $x,y,z\in X$,
\begin{align*} 
E_i^{(2)} A^{*(1)} E_j^{(2)} (x \otimes y \otimes z) = 0.
\end{align*}
Write
$ E_j y = \sum_{\xi \in X} \alpha_{\xi} \xi$  $(\alpha_\xi \in \mathbb C)$.
 We have
\begin{align*}
E_i^{(2)} A^{*(1)} E_j^{(2)} (x \otimes y \otimes z) &= E_i^{(2)} A^{*(1)}  (x \otimes E_j y \otimes z)  \qquad  \qquad \qquad {\hbox{\rm by Lemma \ref{lem:E3action}}} \\
& =    E_i^{(2)} A^{*(1)}   \sum_{\xi \in X}  x\otimes \xi \otimes z \alpha_\xi \\
& =    E_i^{(2)}  \sum_{\xi \in X}  x\otimes \xi \otimes z \alpha_\xi \theta^*_{\partial(\xi, z)}      \qquad \qquad \;{\hbox{\rm by Definition  \ref{def:mapsAAAs}}}       \\
& =    \vert X \vert E_i^{(2)} \Bigl( x \otimes \bigl( E_1 z \circ E_j y \bigr) \otimes z \Bigr)            \quad  {\hbox{\rm by Lemma \ref{lem:dpi} and  \eqref{eq:circMeaning}}}            \\
& =     \vert X \vert   \biggl( x \otimes \Bigl(E_i \bigl( E_1 z \circ E_j y \bigr)\Bigr) \otimes z \biggr)       \qquad  {\hbox{\rm by Lemma \ref{lem:E3action}}   }       \\
&=0 \qquad \qquad \qquad \qquad \qquad \qquad \qquad \quad \; {\hbox{\rm by Lemma \ref{lem:circ}. }}
\end{align*}

\end{proof}

\begin{proposition} \label{prop:AsAAs}
For distinct $r,s \in \lbrace 1,2,3 \rbrace$ we have
\begin{align*}
\bigl \lbrack A^{(r)}, A^{(r)2} A^{*(s)}  - \beta A^{(r)}  A^{*(s)} A^{(r)} + A^{*(s)} A^{(r)2} - \gamma(A^{(r)}  A^{*(s)} +A^{*(s)}A^{(r)}) - \varrho A^{*(s)} \bigr\rbrack &=0.
\end{align*}
\end{proposition}
\begin{proof} Let $C$ denote the expression on the left. We show that $C=0$. We have
\begin{align*}
C = I C I = \Biggl( \sum_{i=0}^D E_i^{(r)} \Biggr) C  \Biggl( \sum_{j=0}^D E_j^{(r)} \Biggr) = \sum_{i=0}^D \sum_{j=0}^D E_i^{(r)}C E_j^{(r)}.
\end{align*}
For $0\leq i,j\leq D$ we show that $E_i^{(r)}C E_j^{(r)}=0$. Using $E_i^{(r)} A^{(r)} = \theta_i E_i^{(r)}$ and  $A^{(r)}E_j^{(r)}  = \theta_j E_j^{(r)}$, we obtain
\begin{align} \label{eq:term}
E_i^{(r)}C E_j^{(r)}=E_i^{(r)} A^{*(s)}  E_j^{(r)} (\theta_i - \theta_j)P(\theta_i, \theta_j),
\end{align}
where the polynomial $P(\lambda, \mu)$ is defined by
\begin{align*}
P(\lambda, \mu) = \lambda^2 - \beta \lambda \mu + \mu^2 - \gamma(\lambda + \mu) - \varrho.
\end{align*}
We examine the factors on the right in \eqref{eq:term}.
If  $\vert i - j \vert >1$ then $E_i^{(r)} A^{*(s)}  E_j^{(r)}=0$ by Lemma  \ref{prop:EAsE}.
If $\vert i-j\vert = 1$ then $P(\theta_i, \theta_j)=0$ by Lemma  \ref{lem:TTR}(iii).
If $i=j$ then of course  $\theta_i - \theta_j=0$.
By these comments, the expression on the right in  \eqref{eq:term} is equal to zero.
We have shown that $E_i^{(r)}C E_j^{(r)}=0$ for $0 \leq i,j\leq D$. Therefore $C=0$.
\end{proof}

\begin{corollary} \label{cor:VVVmod} For the scalars $\beta, \gamma, \gamma^*, \varrho, \varrho^*$ from Lemma \ref{lem:TTR}, the vector space $V^{\otimes 3} $ becomes a $\mathbb T(\beta, \gamma, \gamma^*, \varrho, \varrho^*)$-module on which
\begin{align*}
&A_1 = A^{(1)}, \qquad A_2 = A^{(2)}, \qquad A_3 = A^{(3)},  \\
&A^*_1 = A^{*(1)}, \qquad A^*_2 = A^{*(2)}, \qquad A^*_3 = A^{*(3)}.
\end{align*}
\end{corollary}
\begin{proof} Use Definition \ref{def:symT} along with Lemmas \ref{lem:ABC},  \ref{lem:ABCs} and Propositions  \ref{lem:AASB}, \ref{prop:AAsA},
 \ref{prop:AsAAs}.
\end{proof}

\noindent Theorem \ref{thm:main} follows from Definitions \ref{def:mapsAAA}, \ref{def:mapsAAAs} and Corollary  \ref{cor:VVVmod}.
\medskip

\noindent We end this section with some comments.
\medskip

\noindent The vector space $V$ contains the vector  ${\bf 1} = \sum_{x \in X} x$. By \cite[Section~2]{int} we have
\begin{align}
E_0 x = \vert X \vert^{-1} {\bf 1} \qquad \qquad (x \in X).
\label{eq:E0x}
\end{align}

\begin{lemma} \label{lem:zzz} For distinct $r,s \in \lbrace 1,2,3\rbrace$ the following holds on $V^{\otimes 3}$:
\begin{align*} 
\vert X \vert E_0^{(r)} E_0^{*(s)} E_0^{(r)} = E_0^{(r)}, \qquad \qquad \vert X \vert E_0^{*(r)} E_0^{(s)} E_0^{*(r)} = E_0^{*(r)}.
\end{align*}
\end{lemma}
\begin{proof} By $S_3$-symmetry, we may assume that $r=1$ and $s=2$. We first show that
$\vert X \vert E_0^{(1)} E_0^{*(2)} E_0^{(1)} =E_0^{(1)}$.
Let $x,y,z \in X$. The map $E_0^{(1)}$ sends
\begin{align*}
x \otimes y \otimes z \mapsto \vert X \vert^{-1} {\bf 1} \otimes y \otimes z.
\end{align*}
The map $\vert X \vert E_0^{(1)} E_0^{*(2)} E_0^{(1)}$ sends 
\begin{align*}
{\begin{CD}
x\otimes y \otimes z  @>>\vert X \vert E_0^{(1)}  > {\bf 1} \otimes y \otimes z @>>E_0^{*(2)}> z \otimes y \otimes z @>>E_0^{(1)} > \vert X \vert^{-1} {\bf 1} \otimes y \otimes z.
                                      \end{CD}} 
\end{align*}
\noindent We have shown that 
$\vert X \vert E_0^{(1)} E_0^{*(2)} E_0^{(1)} =E_0^{(1)}$. Next we show that $\vert X \vert E_0^{*(1)} E_0^{(2)} E_0^{*(1)} = E_0^{*(1)}$.
Let $x,y,z \in X$. The map $E_0^{*(1)}$ sends
\begin{align*}
x \otimes y \otimes z \mapsto \delta_{y,z}x  \otimes y \otimes y.
\end{align*}
The map $\vert X \vert E_0^{*(1)} E_0^{(2)} E_0^{*(1)}$ sends 
\begin{align*}
{\begin{CD}
x\otimes y \otimes z  @>> E_0^{*(1)}  > \delta_{y,z} x \otimes y \otimes y @>>\vert X \vert E_0^{(2)}> \delta_{y,z}  x \otimes {\bf 1} \otimes y @>>E_0^{*(1)} > \delta_{y,z} x \otimes y \otimes y.
                                      \end{CD}} 
\end{align*}
\noindent We have shown  that $\vert X \vert E_0^{*(1)} E_0^{(2)} E_0^{*(1)} = E_0^{*(1)}$. The result follows.
\end{proof}

\section{The fundamental $\mathbb T$-module}
 We continue to discuss the $Q$-polynomial distance-regular graph  $\Gamma=(X,\mathcal R)$ from Section 5.  
 Recall the scalars $\beta, \gamma, \gamma^*, \varrho, \varrho^*$ from Lemma \ref{lem:TTR}. Consider the
 standard module $V$ and the vector space $V^{\otimes 3}$.
 In Theorem \ref{thm:main} we turned  $V^{\otimes 3}$ into a module for the algebra $\mathbb T = \mathbb T(\beta, \gamma, \gamma^*, \varrho, \varrho^*)$.
 In this section, we discuss a certain $\mathbb T$-submodule of $V^{\otimes 3}$, said to be fundamental.
 To facilitate this discussion, we bring
in some Hermitean forms.
\medskip

\noindent We define a Hermitean form $( \,,\,) : V \times V \to \mathbb C$ as follows. Pick $u,v\in V$ and write
\begin{align*}
u= \sum_{x \in X} u_x x, \qquad \qquad v = \sum_{x \in X} v_x x, \qquad \qquad u_x, v_x \in \mathbb C.
\end{align*}
Then 
\begin{align*}
(u,v ) =\sum_{x \in X} u_x {\overline v_x},
\end{align*}
where $- $ denotes the complex-conjugate. We abbreviate $\Vert u\Vert^2 = (u,u)$. Note that $(\,,\,)$ is the unique
Hermitean form $V\times V \to \mathbb C$ with respect to which the basis  $X$ is orthonormal.
For $x,y \in X$ we have
\begin{align*}
(Ax,y) = (x,Ay) =  \begin{cases} 1 & {\mbox{\rm if $\partial(x,y)=1$}}; \\
0, & {\mbox{\rm if $\partial(x,y)\not= 1$}}.
\end{cases} 
\end{align*}
Moreover for $u,v \in V$ we have
\begin{align*}
(Au,v) = (u,Av),\qquad \qquad \quad (E_i u,v) = (u,E_i v) \qquad (0 \leq i \leq D).
\end{align*}

\begin{lemma} \label{lem:Herm} The following hold.
\begin{enumerate}
\item[\rm (i)]
There exists a unique Hermitean form $\langle\,,\,\rangle: V^{\otimes 3} \times V^{\otimes 3} \to \mathbb C$ 
with respect to which the basis $X^{\otimes 3}$  is orthonormal.
\item[\rm (ii)]  For $u,v,w,u',v',w' \in V$ we have
\begin{align*} 
\langle u \otimes v \otimes w, u' \otimes v' \otimes w' \rangle = (u,u')( v, v' ) (w, w' ).
\end{align*}
\end{enumerate}
\end{lemma}
\begin{proof} Item (i) is clear. Item (ii) is routinely checked.
\end{proof}

\begin{lemma} \label{lem:VVVH} For $r \in \lbrace 1,2,3\rbrace $ and $u,v \in V^{\otimes 3}$ we have
\begin{align*}
\langle A^{(r)} u,v\rangle = \langle u, A^{(r)} v \rangle, \qquad \qquad
\langle A^{*(r)} u,v\rangle = \langle u, A^{*(r)} v \rangle.
\end{align*}
\end{lemma}
\begin{proof} Without loss of generality, we may assume that $u,v$ are contained in the basis $X^{\otimes 3}$ of $V^{\otimes 3}$.
For such $u,v$ the result is routinely checked.
\end{proof}

\begin{lemma} \label{lem:VVVE} For $r \in \lbrace 1,2,3\rbrace $ and $u,v \in V^{\otimes 3}$ we have
\begin{align*}
\langle E_i^{(r)} u,v\rangle = \langle u, E_i^{(r)} v \rangle, \qquad \quad
\langle E_i^{*(r)} u,v\rangle = \langle u, E_i^{*(r)} v \rangle, \qquad \quad (0 \leq i \leq D).
\end{align*}
\end{lemma}
\begin{proof} By Lemma  \ref{lem:E3prop}, $E_i^{(r)} $ is a polynomial in $A^{(r)}$ that has real coefficients. 
By Lemma  \ref{lem:E3sprop}, $E_i^{*(r)} $ is a polynomial in $A^{*(r)}$ that has real coefficients. 
The result follows in view of Lemma \ref{lem:VVVH}.
\end{proof}

\noindent 
Let $U$ denote a subspace of the vector space $V^{\otimes 3}$. 
 Recall the orthogonal complement
\begin{align*}
U^\perp = \lbrace v \in V^{\otimes 3}  \vert \langle u,v \rangle = 0 \; \forall u \in U\rbrace.
\end{align*}
\noindent By linear algebra, the sum $V^{\otimes 3} = U + U^\perp$ is direct. Let $W$ denote a subspace of $V^{\otimes 3}$ that contains $U$.
By linear algebra, the sum $W = U + U^\perp \cap W$ is direct. We call $U^\perp \cap W$ the {\it orthogonal complement of $U$ in $W$}.

\begin{definition}\rm A $\mathbb T$-module $W$ is called {\it irreducible} whenever $W\not=0$ and $W$ does not contain a  $\mathbb T$-submodule besides $0$ and $W$.
\end{definition}
\begin{lemma} \label{lem:3p} The following {\rm (i)--(iii)} hold.
\begin{enumerate}
\item[\rm (i)]
Let $W$ denote a $\mathbb T$-submodule of  $V^{\otimes 3}$, and let $U$ denote a $\mathbb T$-submodule of $W$.
Then the orthogonal complement of $U$ in $W$ is a $\mathbb T$-submodule.
\item[\rm (ii)] Every nonzero $\mathbb T$-submodule of $V^{\otimes 3}$ is an orthogonal direct sum of irreducible $\mathbb T$-submodules.
\item[\rm (iii)] The $\mathbb T$-module $V^{\otimes 3}$ is an orthogonal direct sum of irreducible $\mathbb T$-submodules.
\end{enumerate}
\end{lemma}
\begin{proof} (i) By Lemma  \ref{lem:VVVH}  and since the $\mathbb T$-generators act on $V^{\otimes 3}$ as $\lbrace A^{(r)}\rbrace_{r=1}^3$, $\lbrace A^{*(r)} \rbrace_{r=1}^3$. \\
\noindent (ii) By (i) and induction on the dimension of the $\mathbb T$-submodule in question. \\
\noindent (iii) This is a special case of (ii).
\end{proof}
\noindent Recall that $V$ contains the vector  ${\bf 1} = \sum_{x \in X} x$. 
By \eqref{eq:E0x} we have $E_0V={\rm Span}({\bf 1})$.
We abbreviate ${\bf 1}^{\otimes 3} = {\bf 1} \otimes {\bf 1}\otimes {\bf 1}$ and note that
\begin{align} \label{eq:one3}
{\bf 1}^{\otimes 3} = \sum_{x,y,z \in X} x \otimes y \otimes z.
\end{align}
We have
\begin{align*}
E_0^{(1)} E_0^{(2)} E_0^{(3)} V^{\otimes 3} = E_0 V \otimes E_0V \otimes E_0V = {\rm Span}({\bf 1}^{\otimes 3}).
\end{align*}

\begin{proposition} There exists a unique irreducible $\mathbb T$-submodule of $V^{\otimes 3}$ that contains  ${\bf 1}^{\otimes 3}$. 
\end{proposition}
\begin{proof} By Lemma  \ref{lem:3p}(iii), the $\mathbb T$-module $V^{\otimes 3}$ is a direct sum of irreducible $\mathbb T$-submodules.
These $\mathbb T$-submodules cannot all be orthogonal to ${\bf 1}^{\otimes 3}$,
so there exists an irreducible $\mathbb T$-submodule $W$ of $V^{\otimes 3}$ that is not orthogonal to ${\bf 1}^{\otimes 3}$.
We have
\begin{align*}
 0 \not= \langle {\bf 1}^{\otimes 3}, W\rangle = \langle E_0^{(1)} E_0^{(2)} E_0^{(3)}{\bf 1}^{\otimes 3}, W\rangle =  \langle {\bf 1}^{\otimes 3}, E_0^{(1)} E_0^{(2)} E_0^{(3)}W\rangle
 \end{align*}
 so $E_0^{(1)} E_0^{(2)} E_0^{(3)}W\not=0$. 
 We have
 \begin{align*}
 0 \not = E_0^{(1)} E_0^{(2)} E_0^{(3)}W  \subseteq   E_0^{(1)} E_0^{(2)} E_0^{(3)} V^{\otimes 3}  = {\rm Span} \bigl({\bf 1}^{\otimes 3}\bigr).
 \end{align*}
 Therefore
 \begin{align*}
 {\bf 1}^{\otimes 3} \in E_0^{(1)} E_0^{(2)} E_0^{(3)}W \subseteq W.
 \end{align*}
 We have shown that ${\bf 1}^{\otimes 3}$ is contained in the irreducible $\mathbb T$-submodule $W$.   Suppose that ${\bf 1}^{\otimes 3}$ is contained in an  irreducible $\mathbb T$-submodule $W'$.
 The $\mathbb T$-submodule  $W\cap W' $ is nonzero since it contains ${\bf 1}^{\otimes 3}$. By this and irreducibility,
 $W = W \cap W' = W'$.
\end{proof}

\begin{definition} \label{def:fun} \rm Let $\Lambda $ denote the unique irreducible $\mathbb T$-submodule of $V^{\otimes 3}$ that contains  ${\bf 1}^{\otimes 3}$.
The $\mathbb T$-submodule  $\Lambda$ is called {\it fundamental}.
\end{definition}

\begin{lemma} \label{lem:gen} We have $\Lambda = \mathbb T ({\bf 1}^{\otimes 3})$. In other words, the $\mathbb T$-module $\Lambda$ is generated by ${\bf 1}^{\otimes 3}$.
\end{lemma}
\begin{proof} The $\mathbb T$-module $\Lambda$ contains ${\bf 1}^{\otimes 3}$, so $\mathbb T ({\bf 1}^{\otimes 3}) \subseteq \mathbb T \Lambda  \subseteq \Lambda$.
The subspace $\mathbb T({\bf 1}^{\otimes 3})$ is a nonzero $\mathbb T$-submodule of $\Lambda$, so  $\mathbb T({\bf 1}^{\otimes 3})=\Lambda$ by the irreducibility of $\Lambda$.
\end{proof}

\noindent  In order to describe $\Lambda$, we will display some vectors contained in $\Lambda$. This is our goal for the rest of the section.
\medskip

\noindent Before we get into the details, we would like to acknowledge that the vectors on display are well known in the context of
Norton algebras  \cite[Section~5]{norton}, 
scaffolds \cite[Theorem~3.8]{scaffold}, and the triple-product relations for the subconstituent algebra \cite[Section~8]{int}.

\begin{definition} \label{def:Phij} For $0 \leq h,i,j\leq D$ define 
\begin{align*}
P_{h,i,j} = \sum_{\stackrel{x,y,z\in X}{       \stackrel{\partial(y,z)=h}         {   \stackrel{\partial(z,x) = i}{\stackrel{ \partial(x,y)=j}{}}       }       }} x \otimes y \otimes z.
\end{align*}
\end{definition}

\begin{lemma} \label{lem:PLam}
The following hold for $0 \leq h,i,j\leq D$:
\begin{enumerate}
\item[\rm (i)] $P_{h,i,j} = E_h^{*(1)} E_i^{*(2)} E_j^{*(3)} ({\bf 1}^{\otimes 3})$;
\item[\rm (ii)] $P_{h,i,j} \in \Lambda$.
\end{enumerate}
\end{lemma}
\begin{proof} (i) Use Lemma \ref{lem:E3saction} and  \eqref{eq:one3}. \\
\noindent (ii) By (i) and since ${\bf 1}^{\otimes 3} \in \Lambda$.
\end{proof}

\noindent Recall from \cite[Section~2]{int} the valencies $k_i $ $(0 \leq i \leq D)$.
\begin{lemma} \label{lem:Pnorm} The following vectors are mutually orthogonal:
\begin{align*}
P_{h,i,j} \qquad \qquad 0 \leq h,i,j\leq D.
\end{align*} For $0 \leq h,i,j\leq D$ we have
\begin{align} \label{eq:Pnorm}
\Vert P_{h,i,j} \Vert^2 = \vert X \vert k_h p^h_{i,j}.
\end{align}
\end{lemma}
\begin{proof} By Definition  \ref{def:Phij} and since $X^{\otimes 3}$
is an orthonormal basis for $V^{\otimes 3}$.
\end{proof}

\begin{lemma} \label{lem:Pzero}
$P_{h,i,j} =0$ if and only if $p^h_{i,j} =0$   $(0 \leq h,i,j\leq D)$.
\end{lemma}
\begin{proof} Immediate from \eqref{eq:Pnorm}.
\end{proof}

\begin{lemma} \label{lem:Pfacts}
We have
\begin{align} \label{eq:Psum}
{\bf 1}^{\otimes 3}= \sum_{h=0}^D \sum_{i=0}^D \sum_{j=0}^D P_{h,i,j}.
\end{align}
Moreover,
\begin{align}
P_{0,0,0} = \sum_{x \in X} x \otimes x \otimes x.
\label{eq:P000}
\end{align}
\end{lemma}
\begin{proof} To obtain \eqref{eq:Psum}, use  \eqref{eq:one3} and Definition \ref{def:Phij}. To obtain
\eqref{eq:P000}, set $h=i=j=0$ in Definition \ref{def:Phij}.
\end{proof}

\begin{definition} \label{def:Qhij}  {\rm (See \cite[Section~5]{norton}.)}  For $0 \leq h,i,j\leq D$ define 
\begin{align*}
Q_{h,i,j} = \vert X \vert \sum_{x \in X} E_h x \otimes E_i x \otimes E_j x.
\end{align*}
\end{definition}

\begin{lemma} \label{lem:QLam}
The following hold for $0 \leq h,i,j\leq D$:
\begin{enumerate}
\item[\rm (i)] $Q_{h,i,j} =\vert X \vert E_h^{(1)} E_i^{(2)} E_j^{(3)} (P_{0,0,0})$;
\item[\rm (ii)] $Q_{h,i,j} \in \Lambda$.
\end{enumerate}
\end{lemma}
\begin{proof} (i) Use Lemma \ref{lem:E3action}  and \eqref{eq:P000}.\\
\noindent (ii) By (i) and since $P_{0,0,0} \in \Lambda$ by Lemma  \ref{lem:PLam}(ii).
\end{proof}

\noindent  Recall from \cite[p.~64]{banIto} or \cite[Section~5]{int} the Krein parameters $q^h_{i,j}$ $(0 \leq h,i,j\leq D)$.
Define $m_h = {\rm dim}(E_hV)$ for $0 \leq h \leq D$.

\begin{lemma} \label{lem:Qnorm} {\rm (See \cite[Lemma~4.2]{norton}.)} The following vectors are mutually orthogonal:
\begin{align*}
Q_{h,i,j} \qquad \qquad 0 \leq h,i,j\leq D.
\end{align*} For $0 \leq h,i,j\leq D$ we have
\begin{align} \label{eq:Qnorm}
\Vert Q_{h,i,j} \Vert^2 = \vert X \vert m_h q^h_{i,j}.
\end{align}
\end{lemma}

\begin{lemma} \label{lem:Qzero}
$Q_{h,i,j} =0$ if and only if $q^h_{i,j} =0$   $(0 \leq h,i,j\leq D)$.
\end{lemma}
\begin{proof} Immediate from \eqref{eq:Qnorm}.
\end{proof}

\begin{lemma} \label{lem:Qfacts}
We have
\begin{align} \label{eq:Qsum}
 P_{0,0,0}= \vert X \vert^{-1} \sum_{h=0}^D \sum_{i=0}^D \sum_{j=0}^D Q_{h,i,j}.
\end{align}
Moreover,
\begin{align}
Q_{0,0,0} =\vert X \vert^{-1} {\bf 1}^{\otimes 3}.
\label{eq:Q000}
\end{align}
\end{lemma}
\begin{proof}  To obtain \eqref{eq:Qsum}, observe that
\begin{align*}
\vert X \vert^{-1} \sum_{h=0}^D \sum_{i=0}^D \sum_{j=0}^D Q_{h,i,j} &=   \sum_{h=0}^D \sum_{i=0}^D \sum_{j=0}^D \sum_{x \in X} E_h x \otimes E_i x \otimes E_j x \\
&= \sum_{x \in X}  \sum_{h=0}^D \sum_{i=0}^D \sum_{j=0}^D  E_h x \otimes E_i x \otimes E_j x  \\
&=\sum_{x \in X} \Biggl( \sum_{h=0}^D E_h x \Biggr) \otimes \Biggl( \sum_{i=0}^D E_i x \Biggr) \otimes \Biggl( \sum_{j=0}^D E_j x \Biggr) \\
&= \sum_{x \in X} x \otimes x \otimes x \\
&=  P_{0,0,0}.
\end{align*}
\noindent To obtain \eqref{eq:Q000}, set $h=i=j=0$ in Definition  \ref{def:Qhij} and use \eqref{eq:E0x}.
\end{proof}

\section{Two commuting actions}
 We continue to discuss the $Q$-polynomial distance-regular graph  $\Gamma=(X,\mathcal R)$ from Section 5.
 Recall the scalars $\beta, \gamma, \gamma^*, \varrho, \varrho^*$ from Lemma \ref{lem:TTR}. Consider the
 standard module $V$ and the vector space $V^{\otimes 3}$.
 In Theorem \ref{thm:main} we turned  $V^{\otimes 3}$ into a module for the algebra $\mathbb T = \mathbb T(\beta, \gamma, \gamma^*, \varrho, \varrho^*)$.
 In this section, we consider a subgroup $G$ of the automorphism group of $\Gamma$. We describe how $V^{\otimes 3}$ becomes a $G$-module. We show that the  $G$ action on $V^{\otimes 3}$ commutes 
 with the  $\mathbb T$ action on $V^{\otimes 3}$. We use the $G$ action on $V^{\otimes 3}$ to describe the fundamental $\mathbb T$-submodule $\Lambda$.
\medskip

\noindent By an {\it automorphism of $\Gamma$} we mean a permutation $g$ of $X$ such that for all $x,y \in X$,
\begin{align*}
\hbox{\rm  $x,y$ are adjacent if and only if $g(x), g(y)$ are adjacent.}
\end{align*}
A permutation $g$ of $X$ is an automorphism of $\Gamma$ if and only if $\partial(x,y) = \partial \bigl( g(x), g(y) \bigr) $ for all $x,y \in X$.
The automorphism group ${\rm Aut}(\Gamma)$ consists of the automorphisms of $\Gamma$; the group operation is composition. Throughout this section, let $G$
denote a subgroup of ${\rm Aut}(\Gamma)$.
\medskip

\noindent Let us recall how $V$ becomes a $G$-module. 
Pick $v \in V$ and write $v = \sum_{x \in X} v_x x$ $(v_x \in \mathbb C)$. For all $g \in G$,
\begin{align*}
g(v) = \sum_{x \in X} v_x g(x).
\end{align*} 
Since $g$ respects adjacency, we have $gA=Ag$ on $V$.
\medskip

\noindent Next, we describe how $V^{\otimes 3}$ becomes a $G$-module. For $u,v,w \in V$ and $g \in G$ we have
\begin{align*}
g(u \otimes v \otimes w) = g(u) \otimes g(v) \otimes g(w).
\end{align*}

\begin{proposition} \label{prop:GTcom} For $g \in G $ and $B \in \mathbb T$, we have $gB=Bg$ on $V^{\otimes 3}$.
\end{proposition}
\begin{proof} It suffices to show that  for $r \in \lbrace 1,2,3\rbrace$ the following holds on $V^{\otimes 3}$:
\begin{align*} 
g A^{(r)} = A^{(r)} g, \qquad \qquad g A^{*(r)} = A^{*(r)} g.
\end{align*}
These equations are routinely checked using Lemma \ref{lem:Atensor}  and Definition \ref{def:mapsAAAs}.
\end{proof}

\begin{definition} \label{lem:fix} \rm We define the set
\begin{align*}
{\rm Fix}(G) = \lbrace v \in V^{\otimes 3} \vert g(v)=v \;\forall g \in G \rbrace.
\end{align*}
\end{definition}

\begin{lemma} \label{lem:FixGsm} 
${\rm Fix}(G)$ is a $\mathbb T$-submodule of $V^{\otimes 3}$.
\end{lemma}
\begin{proof} We first check that ${\rm Fix}(G)$ is a subspace of the vector space $V^{\otimes 3}$. This holds by
Definition \ref{lem:fix}
and since each element of $G$ acts on $V^{\otimes 3}$ in $\mathbb C$-linear fashion.
Next, we check that ${\rm Fix}(G)$ is invariant under $\mathbb T$. For $B \in \mathbb T$ and $v \in {\rm Fix}(G)$
we show that $Bv \in {\rm Fix}(G)$. Let $g \in G$. Using $g(v)=v$ and Proposition \ref{prop:GTcom}, we obtain
\begin{align*}
g(Bv) = gB(v) = Bg(v) = Bv.
\end{align*}
Therefore $Bv \in {\rm Fix}(G)$. We have shown that ${\rm Fix}(G)$ is invariant under $\mathbb T$. The result follows.
\end{proof}

\noindent Next, we describe how ${\rm Fix}(G)$ is related to the fundamental $\mathbb T$-submodule $\Lambda$.

\begin{proposition} \label{prop:LamSub} We have $\Lambda \subseteq {\rm Fix}(G)$.
\end{proposition}
\begin{proof} For all $g \in G$ we have
$g({\bf 1}) = {\bf 1}$, because ${\bf 1} = \sum_{x \in X} x$ and $g$ permutes $X$.
We have ${\bf 1}^{\otimes 3} \in {\rm Fix}(G)$, because for all $g \in G$,
\begin{align*}
g({\bf 1}^{\otimes 3}) =g({\bf 1} \otimes  {\bf 1} \otimes  {\bf 1} )= g({\bf 1}) \otimes  g({\bf 1}) \otimes  g({\bf 1}) = {\bf 1} \otimes  {\bf 1} \otimes  {\bf 1} = {\bf 1}^{\otimes 3}.
\end{align*}
By these comments and Lemmas \ref{lem:gen}, \ref{lem:FixGsm} we obtain
\begin{align*}
\Lambda = \mathbb T ({\bf 1}^{\otimes 3})  \subseteq \mathbb T \,{\rm Fix}(G) \subseteq {\rm Fix}(G).
\end{align*}
\end{proof}
\noindent Next, we display an orthogonal basis for ${\rm Fix}(G)$. We will use the following notation. Recall that $V^{\otimes 3}$
has an orthonormal basis $X^{\otimes 3}$. The group $G$ acts on the set $X^{\otimes 3}$.
 
 \begin{definition}\label{def:Gact} \rm Referring to the $G$ action on the set $X^{\otimes 3}$,
 let $\mathcal O$ denote the set of orbits. For each orbit $\Omega \in \mathcal O$ define
 \begin{align*}
 \chi_\Omega = \sum_{x \otimes y \otimes z \in \Omega} x \otimes y \otimes z.
 \end{align*}
 We call $\chi_\Omega$ the {\it characteristic vector} of $\Omega$.
 \end{definition}

 \begin{proposition} \label{prop:orthBasis} The following is an orthogonal basis for the vector space ${\rm Fix}(G)$:
 \begin{align} \label{eq:OrthBasis}
 \chi_\Omega, \qquad \quad \Omega \in \mathcal O.
 \end{align}
 \end{proposition}
 \begin{proof} By construction, the vectors  \eqref{eq:OrthBasis} are contained in ${\rm Fix}(G)$.
 It is routine to check that the vectors \eqref{eq:OrthBasis} span ${\rm Fix}(G)$.
 The vectors \eqref{eq:OrthBasis} are mutually orthogonal because the $G$-orbits in $\mathcal O$ are mutually disjoint subsets of $X^{\otimes 3}$,
  and because the vectors in $X^{\otimes 3}$  are
 mutually orthogonal.
 The vectors \eqref{eq:OrthBasis} are linearly independent, because they are nonzero and mutually orthogonal. The result follows.
 \end{proof}
 \begin{corollary} The dimension of ${\rm Fix}(G) $ is equal to $\vert \mathcal O \vert$.
 \end{corollary}
 \begin{proof} Immediate from Proposition \ref{prop:orthBasis}.
 \end{proof}

\noindent We showed in Proposition \ref{prop:LamSub} that  $\Lambda \subseteq  {\rm Fix}(G)$. It sometimes happens that $\Lambda = {\rm Fix}(G)$. We will give an
example in the next section.

\section{The Hamming graph $H(D,N)$}

\noindent Recall the $Q$-polynomial distance-regular graph $\Gamma=(X,\mathcal R)$ from Section 5.
 In this section, we assume that $\Gamma$ is a Hamming graph 
 \cite{bbit,
 banIto,
  fermion2,
 bcn,
 delsarte,
 tanakaCode,
 hwh}. Under this assumption, we will describe the fundamental $\mathbb T$-module $\Lambda$ from Definition  \ref{def:fun}.

\begin{example}\label{ex:hypercube} \rm (See  \cite[Chapter~6.4]{bbit}, \cite[Chapter 3.2]{banIto}, \cite[Section 9.2]{bcn}.)
For  integers $D\geq 1$ and $N\geq 3$, the {\it Hamming graph} $H(D,N)$  has vertex set $X$ consisting of the $D$-tuples of elements taken from the set $\lbrace 1,2,\ldots, N\rbrace$. 
Vertices $x,y \in X$ are adjacent  whenever $x,y$ differ in exactly one coordinate. The graph $H(D,N)$ is distance-regular with diameter $D$ and intersection numbers
\begin{align*}
p^i_{1,i-1} &= i, \qquad \qquad (1 \leq i \leq D),\\
p^i_{1,i+1} &= (N-1)(D-i) \qquad \qquad (0 \leq i \leq D-1). 
\end{align*}
The graph $H(D,N)$ is $Q$-polynomial with
\begin{align}
\theta_i = \theta^*_i = D(N-1)-i N \qquad \qquad (0 \leq i \leq D). \label{eq:hamEig}
\end{align}
\end{example}

\noindent Throughout this section, we assume that $\Gamma=H(D,N)$. Note that $\vert X \vert = N^D$. By \cite[Section 9.2]{bcn}, for $x, y \in X$ the distance $\partial(x,y)$ is equal to the number of coordinates at which $x,y$ differ.
The parameters from Lemma \ref{lem:TTR}  are
\begin{align*}
\beta = 2, \qquad \qquad \gamma=\gamma^*=0, \qquad \qquad \varrho = \varrho^*=N^2.
\end{align*}
By \cite[Theorem~2.86]{bbit}, we have $p^h_{i,j} = q^h_{i,j}$ for $0 \leq h,i,j\leq D$.
By \cite[Theorem~9.2.1]{bcn} the automorphism group ${\rm Aut}(\Gamma)$ is isomorphic to the wreath product of the symmetric group $S_N$ and
the symmetric group $S_D$.
The elements of $S_N$ permute the set $\lbrace 1,2,\ldots, N\rbrace$
and the elements of $S_D$ permute the vertex coordinates.
By \cite[p.~207]{banIto} 
the graph $\Gamma$ is distance-transitive in the sense of \cite[p.~189]{banIto}. We take $G={\rm Aut}(\Gamma)$.
\medskip

\noindent Recall that $G$ acts on the set $X^{\otimes 3}$, and $\mathcal O$ is the set of orbits. Our next goal is to describe $\mathcal O$.
It is shown in \cite[Propositions~1,~2]{tanakaCode} that $\vert \mathcal O \vert = \binom{D+4}{4}$. We will give a short proof, for the sake of completeness and to set up some notation. 
We introduce a set $\mathcal P$ with cardinality $\binom{D+4}{4}$, and display a bijection
$\mathcal O \to \mathcal P$.

\begin{definition}\label{def:profile} \rm Let $\mathcal P$ denote the set of sequences $(d_1, d_2, d_3, d_4, d_5)$ of natural numbers such that $\sum_{i=1}^5 d_i=D$.
Elements of $\mathcal P$ are called {\it profiles}.
\end{definition}

\begin{lemma} \label{lem:Psize} We have
\begin{align*}
\vert \mathcal P \vert = \binom{D+4}{4}.
\end{align*}
\end{lemma}
\begin{proof} Exercise.
\end{proof}

\begin{definition} \label{lem:orb3} \rm We define a function $f: X^{\otimes 3} \to {\mathcal P}$ as follows.
 Pick $x,y,z \in X$ and write
\begin{align*}
x = (x_1, x_2, \ldots, x_D), \qquad \quad 
y = (y_1, y_2, \ldots, y_D), \qquad \quad
z= (z_1, z_2, \ldots, z_D).
\end{align*}
The function $f$ sends 
\begin{align*}
x \otimes y \otimes z \mapsto 
(d_1, d_2, d_3, d_4, d_5),
\end{align*}
where
\begin{align*}
d_1 &= \bigl \vert \lbrace i \vert 1 \leq i \leq D, \; x_i=y_i=z_i \rbrace \bigr \vert, \\
d_2 &= \bigl \vert \lbrace i \vert 1 \leq i \leq D, \; x_i \not=y_i =z_i \rbrace \bigr \vert , \\
d_3&= \bigl \vert \lbrace i \vert 1 \leq i \leq D, \;  y_i \not=z_i=x_i \rbrace \bigr \vert, \\
d_4 &= \bigl \vert \lbrace i \vert 1 \leq i \leq D, \; z_i \not=x_i =y_i \rbrace \big\vert, \\
d_5 &= \bigl \vert \lbrace i \vert 1 \leq i \leq D, \; x_i\not=y_i\not=z_i \not=x_i \rbrace \bigr \vert.
\end{align*}
We call $(d_1, d_2, d_3, d_4, d_5)$  the {\it profile of $x \otimes y \otimes z$}.
\end{definition}

\begin{lemma} For a profile $(d_1, d_2, d_3, d_4, d_5) \in \mathcal P$, the number of vectors in $X^{\otimes 3}$ with this profile is equal to
\begin{align} \label{eq:preimage}
\frac{D!}{d_1! d_2!d_3!d_4!d_5!} N^D (N-1)^{D-d_1}(N-2)^{d_5}.
\end{align}
\end{lemma}
\begin{proof} By combinatorial counting.
\end{proof}

\begin{corollary} \label{cor:surj} The  function $f: X^{\otimes 3} \to {\mathcal P}$ is surjective.
\end{corollary}
\begin{proof} The numbers \eqref{eq:preimage} are all nonzero.
\end{proof}

\noindent The following result is a variation on \cite[Proposition~2]{tanakaCode}.
\begin{lemma} \label{lem:orb} A pair of vectors in $X^{\otimes 3}$ are in the same $G$-orbit if and only if they have the same profile. 
\end{lemma}
\begin{proof} This is routinely checked.
\end{proof}

\begin{definition}\label{eq:fadj} \rm We define a function $F: \mathcal O \to \mathcal P$ as follows.
For $\Omega \in \mathcal O$ define $F(\Omega) = f(x \otimes y \otimes z)$, where $x \otimes y \otimes z$
is any vector in $\Omega$. 
We call $F(\Omega)$ the {\it profile of $\Omega$}.
\end{definition}

\begin{proposition} \label{prop:Fbij} The function  $F:\mathcal O \to \mathcal P$ is a bijection. Moreover,
\begin{align}\label{eq:Osize}
\vert \mathcal O\vert = \binom{D+4}{4}.
\end{align}
\end{proposition}
\begin{proof} The function $F$ is a bijection by Corollary \ref{cor:surj} and Lemma \ref{lem:orb}. The equation \eqref{eq:Osize}
is from Lemma \ref{lem:Psize}.
\end{proof}

\noindent We have a comment.
\begin{lemma}\label{lem:dsum}
A vector $x\otimes y \otimes z$ in $X^{\otimes 3}$ with profile $(d_1,d_2,d_3,d_4,d_5)$ satisfies
\begin{align*}
\partial(x,y) &= d_2+d_3+d_5, \\
\partial(y,z) &= d_3+d_4+d_5,\\
\partial(z,x) &= d_4+d_2+d_5.
\end{align*}
\end{lemma}
\begin{proof} The distance between two given vertices is equal to the number of coordinates at which they differ. The result follows in view of  Definition  \ref{lem:orb3}.
\end{proof}

\noindent Our next general goal is to show that $\Lambda={\rm Fix}(G)$.  To this end, we introduce some notation.
\begin{definition} \label{def:xiP} \rm For each profile $(d_1, d_2, d_3, d_4, d_5) \in \mathcal P$ let
$\chi (d_1, d_2, d_3, d_4, d_5 )$ denote the characteristic vector of the corresponding orbit in $\mathcal O$. For notational convenience, define $\chi ( d_1, d_2, d_3, d_4, d_5)=0$ for
any sequence $(d_1, d_2, d_3, d_4, d_5)$ that is not in $\mathcal P$.
\end{definition}

\begin{lemma} \label{lem:Aaction} For a profile $(d_1, d_2, d_3, d_4, d_5) \in \mathcal P$, the following {\rm (i)--(iii)} hold.
\begin{enumerate}
\item[\rm (i)] The vector
\begin{align*}
A^{(1)} \chi (d_1, d_2, d_3, d_4, d_5)
\end{align*}
 is a linear combination with the following terms and coefficients:
\begin{align*} 
\begin{tabular}[t]{c|c}
{\rm Term }& {\rm Coefficient} 
 \\
 \hline
   $  \chi (d_1+1, d_2-1, d_3, d_4, d_5 )$& $(d_1+1) (N-1)$   \\
 $ \chi (d_1-1, d_2+1, d_3, d_4, d_5 )  $   & $d_2+1 $\\
  $ \chi (d_1, d_2, d_3+1, d_4-1, d_5) $& $d_3+1$\\
   $ \chi (d_1, d_2, d_3-1, d_4+1, d_5 )$  & $d_4+1 $ \\
 $ \chi (d_1, d_2, d_3+1, d_4, d_5-1 )$& $(d_3+1) (N-2) $\\
  $  \chi (d_1, d_2, d_3-1, d_4, d_5+1 )$ &$ d_5+1$  \\
   $  \chi(d_1, d_2, d_3, d_4+1, d_5-1)$  & $(d_4+1)(N-2)$ \\
   $  \chi (d_1, d_2, d_3, d_4-1, d_5+1) $ & $d_5+1 $\\
 $        \chi ( d_1, d_2, d_3, d_4, d_5  )$& $d_2 (N-2)+ d_5(N-3) $
   \end{tabular}
\end{align*}
\item[\rm (ii)] 
the vector
\begin{align*}
A^{(2)} \chi (d_1, d_2, d_3, d_4, d_5)
\end{align*}
 is a linear combination with the following terms and coefficients:
\begin{align*} 
\begin{tabular}[t]{c|c}
{\rm Term }& {\rm Coefficient} 
 \\
 \hline
   $  \chi (d_1+1, d_2, d_3-1, d_4, d_5 )$& $(d_1+1) (N-1)$   \\
 $ \chi (d_1-1, d_2,d_3+1, d_4,  d_5 )  $   & $d_3+1 $\\
  $ \chi (d_1, d_2-1, d_3, d_4+1, d_5) $& $d_4+1$\\
   $ \chi (d_1, d_2+1,d_3, d_4-1, d_5 )$  & $d_2+1 $ \\
 $ \chi (d_1, d_2, d_3, d_4+1,  d_5-1 )$& $(d_4+1) (N-2) $\\
   $  \chi (d_1, d_2, d_3, d_4-1,  d_5+1 )$ &$ d_5+1$  \\
   $  \chi(d_1, d_2+1,d_3, d_4, d_5-1)$  & $(d_2+1)(N-2)$ \\
   $  \chi (d_1, d_2-1, d_3, d_4,  d_5+1) $ & $d_5+1 $\\
 $        \chi ( d_1, d_2, d_3, d_4, d_5  )$& $d_3 (N-2)+ d_5(N-3) $
   \end{tabular}
\end{align*}
\newpage
\item[\rm (iii)] 
the vector
\begin{align*}
A^{(3)} \chi (d_1, d_2, d_3, d_4, d_5)
\end{align*}
 is a linear combination with the following terms and coefficients:
\begin{align*} 
\begin{tabular}[t]{c|c}
{\rm Term }& {\rm Coefficient} 
 \\
 \hline
   $  \chi (d_1+1, d_2, d_3, d_4-1,  d_5 )$& $(d_1+1) (N-1)$   \\
 $ \chi (d_1-1, d_2, d_3,d_4+1,   d_5 )  $   & $d_4+1 $\\
  $ \chi (d_1, d_2+1, d_3-1, d_4,  d_5) $& $d_2+1$\\
   $ \chi (d_1, d_2-1, d_3+1,d_4,  d_5 )$  & $d_3+1 $ \\
 $ \chi (d_1, d_2+1,d_3, d_4,   d_5-1 )$& $(d_2+1) (N-2) $\\
    $  \chi (d_1, d_2-1, d_3, d_4,   d_5+1 )$ &$ d_5+1$  \\
   $  \chi(d_1, d_2, d_3+1,d_4,  d_5-1)$  & $(d_3+1)(N-2)$ \\
   $  \chi (d_1, d_2, d_3-1, d_4,  d_5+1) $ & $d_5+1 $\\
 $        \chi ( d_1, d_2, d_3, d_4, d_5  )$& $d_4 (N-2)+ d_5(N-3) $
   \end{tabular}
\end{align*}
\end{enumerate}
\end{lemma}
\begin{proof} By combinatorial counting. \\
\end{proof} 

\begin{lemma}  \label{lem:Asaction} For a profile $(d_1, d_2, d_3, d_4, d_5) \in \mathcal P$, the vector
$\chi (d_1, d_2, d_3, d_4, d_5)$ is a common eigenvector for $A^{*(1)}$, $A^{*(2)}$, $A^{*(3)}$ with
eigenvalues
\begin{align*}
N(d_1+d_2)-D, \qquad \quad N(d_1+d_3)-D, \qquad \quad N(d_1+d_4)-D,
\end{align*}
respectively.
\end{lemma}
\begin{proof} First consider $A^{*(1)}$. By Definition \ref{def:mapsAAAs} and Lemma \ref{lem:dsum}, the vector  $\chi (d_1, d_2, d_3, d_4, d_5)$ is an eigenvector for $A^{*(1)}$ with
eigenvalue $\theta^*_i$, where $i = d_3+d_4+d_5$. Using \eqref{eq:hamEig} we obtain
\begin{align*}
\theta^*_i = N(d_1+d_2)-D.
\end{align*}
The matrices $A^{*(2)}$, $A^{*(3)}$ are similarly treated.
\end{proof}

\begin{proposition}\label{prop:LamF} We have $\Lambda={\rm Fix}(G)$.
\end{proposition}
\begin{proof} By Propositions \ref{prop:orthBasis}, \ref{prop:Fbij} and Definition \ref{def:xiP},  the following vectors form an orthogonal basis for ${\rm Fix}(G)$:
\begin{align} \label{eq:FixBasis}
\chi(d_1,d_2,d_3,d_4,d_5), \qquad \qquad (d_1, d_2, d_3, d_4, d_5) \in \mathcal P.
\end{align}
 By Proposition \ref{prop:LamSub}, $\Lambda\subseteq {\rm Fix}(G)$. To show that equality holds, 
 it suffices to show that $\chi(d_1, d_2, d_3, d_4, d_5) \in \Lambda$ for every profile $(d_1, d_2, d_3, d_4, d_5)$.
We say that a  profile $(d_1, d_2, d_3, d_4, d_5)$ is {\it confirmed} whenever  $\chi(d_1, d_2, d_3, d_4, d_5) \in \Lambda$.
We show that every profile is confirmed. The profile $(D,0,0,0,0)$ is confirmed, because $\chi (D,0,0,0,0)=P_{0,0,0}$ by Definitions \ref{lem:orb3}, \ref{def:xiP} along with  \eqref{eq:P000},
and $P_{0,0,0} \in \Lambda$ by Lemma \ref{lem:PLam}. Let $r \in \lbrace 1,2,3\rbrace$.
 Two distinct profiles $(d_1,d_2,d_3,d_4,d_5)$ and $(d'_1,d'_2,d'_3,d'_4,d'_5)$ will be called
{\it $r$-adjacent} whenever $A^{(r)} \chi(d_1,d_2,d_3,d_4,d_5)$ and $\chi(d'_1,d'_2,d'_3,d'_4,d'_5)$ are not orthogonal. This occurs if and only
if  $\chi(d'_1,d'_2,d'_3,d'_4,d'_5)$ is a term in the first eight rows of the $r$th table of Lemma  \ref{lem:Aaction}. If the profile $(d_1,d_2,d_3,d_4,d_5)$ is confirmed and
the profile $(d'_1,d'_2,d'_3,d'_4,d'_5)$ is $r$-adjacent to it, then  $(d'_1,d'_2,d'_3,d'_4,d'_5)$ is confirmed because for each table of
 Lemma  \ref{lem:Aaction} the terms lie in different common eigenspaces for $A^{*(1)}$, $A^{*(2)}$, $A^{*(3)}$. By these comments, the set of vectors
 \begin{align} \label{eq:C}
 \chi(d_1, d_2, d_3, d_4, d_5), \qquad \qquad (d_1, d_2, d_3, d_4,d_5) \;\mbox{\rm a confirmed profile}
 \end{align}
 span a nonzero subspace of $\Lambda$ that is invariant under each of $A^{(1)}$, $A^{(2)}$, $A^{(3)}$. Let us call this subspace $C$.
 By Lemma  \ref{lem:Asaction},
 each vector in \eqref{eq:C} is a common eigenvector for $A^{*(1)}$, $A^{*(2)}$, $A^{*(3)}$. Therefore $C$ is invariant under each of
 $A^{*(1)}$, $A^{*(2)}$, $A^{*(3)}$. By these comments, $C$ is a 
 $\mathbb T$-submodule of $\Lambda$. We have $C=\Lambda$  since the $\mathbb T$-module $\Lambda$ is irreducible.
 Now consider the vector ${\bf 1}^{\otimes 3}$.
 By construction
 \begin{align*}
 {\bf 1}^{\otimes 3} = \sum \chi(d_1,d_2,d_3,d_4,d_5),
 \end{align*}
 where the sum is over all the profiles $(d_1, d_2, d_3, d_4, d_5)$. 
 We have ${\bf 1}^{\otimes 3} \in \Lambda$ by Definition \ref{def:fun}. Therefore ${\bf 1}^{\otimes 3}$ is a linear combination of
 the vectors  \eqref{eq:C}. By these comments and since the vectors \eqref{eq:FixBasis} are linearly independent, we see that every profile is confirmed.
 We have shown that  $\Lambda={\rm Fix}(G)$.
\end{proof}

\begin{corollary}\label{cor:basisLam} The vector space $\Lambda$ has an orthogonal basis
\begin{align*}
\chi(d_1,d_2,d_3,d_4,d_5), \qquad \qquad (d_1,d_2,d_3,d_4,d_5) \in \mathcal P.
\end{align*}
Moreover,
\begin{align*}
{\rm dim} \, \Lambda = \binom{D+4}{4}.
\end{align*}
\end{corollary}
\begin{proof} By Proposition \ref{prop:LamF} and the construction.
\end{proof}

\noindent We finish this paper with a very special case. Referring to our Hamming graph $\Gamma=H(D,N)$, we now assume that $D=1$. In this case, $\Gamma$ becomes the complete graph $K_N$.
Setting $D=1$ in \eqref{eq:hamEig}, we obtain
\begin{align*}
\theta_0 = \theta^*_0 = N-1, \qquad \qquad \theta_1 = \theta^*_1 = -1.
\end{align*}
We have $\vert \mathcal P\vert = \binom{5}{4}=5$. The elements of $\mathcal P$ are
\begin{align*}
(1,0,0,0,0), \qquad (0,1,0,0,0), \qquad (0,0,1,0,0), \qquad (0,0,0,1,0), \qquad (0,0,0,0,1).
\end{align*}
Using Definition  \ref{def:Phij} and Lemma \ref{lem:dsum}, we obtain
\begin{align*}
& \chi(1,0,0,0,0) = P_{0,0,0}, \qquad \quad \chi(0,1,0,0,0)= P_{0,1,1}, \qquad \quad
 \chi(0,0,1,0,0) = P_{1,0,1}, \\
 & \chi(0,0,0,1,0)=P_{1,1,0}, \qquad \quad  \chi(0,0,0,0,1)= P_{1,1,1}.
\end{align*}
\noindent By this and Corollary \ref{cor:basisLam},
the following is an orthogonal basis for $\Lambda$:
\begin{align} \label{eq:5}
P_{0,0,0}, \qquad P_{0,1,1}, \qquad P_{1,0,1}, \qquad P_{1,1,0}, \qquad P_{1,1,1}.
\end{align}
Using Lemma \ref{lem:Aaction} we find that
with respect to the basis  \eqref{eq:5}, the matrices representing $A^{(1)}$, $A^{(2)}$, $A^{(3)}$ are
\begin{align*}
&A^{(1)}: \quad  \begin{pmatrix} 0 & N-1 & 0&0&0 \\
                                                  1&N-2 &0&0&0 \\
                                                 0 &0&0&1&N-2 \\
                                                 0 &0&1&0&N-2 \\
                                                 0 &0&1&1& N-3
                         \end{pmatrix}, \\
 &A^{(2)}: \quad  \begin{pmatrix} 0 & 0 & N-1&0&0 \\
                                                  0&0 &0&1&N-2\\
                                                 1 &0&N-2&0&0 \\
                                                 0 &1&0&0&N-2 \\
                                                 0 &1&0&1& N-3
                         \end{pmatrix}, \\          
 &A^{(3)}: \quad  \begin{pmatrix} 0 & 0 & 0&N-1&0 \\
                                                  0&0 &1&0&N-2\\
                                                 0 &1&0&0&N-2 \\
                                                 1 &0&0&N-2&0 \\
                                                 0 &1&1&0& N-3
                         \end{pmatrix}.                                                                                  
\end{align*}
\noindent Using Lemma \ref{lem:Asaction} we find that with respect to the basis \eqref{eq:5}, the matrices representing  $A^{*(1)}$, $A^{*(2)}$, $A^{*(3)}$ are
\begin{align*}
& A^{*(1)}: \quad {\rm diag} \bigl( N-1, N-1, -1,-1,-1), \\
& A^{*(2)}: \quad {\rm diag} \bigl( N-1, -1, N-1,-1,-1), \\
& A^{*(3)}: \quad {\rm diag} \bigl( N-1, -1, -1,N-1,-1).
\end{align*}

\noindent For the sake of completeness, we mention another  basis for $\Lambda$.
By Lemmas \ref{lem:QLam}--\ref{lem:Qzero} and  ${\rm dim}\,\Lambda=5$,  the following is an orthogonal basis for $\Lambda$:
\begin{align} \label{eq:5s}
Q_{0,0,0}, \qquad Q_{0,1,1}, \qquad Q_{1,0,1}, \qquad Q_{1,1,0}, \qquad Q_{1,1,1}.
\end{align}
\noindent 
Define $S\in {\rm End}(\Lambda)$ that sends the basis
\eqref{eq:5} to the basis \eqref{eq:5s}. 
One checks (or see \cite[Section~7.1]{Jaeger2}) that with respect to the basis \eqref{eq:5}, the matrix representing $S$ is
\begin{align*}
S: \quad \frac{1}{N}  \begin{pmatrix} 1 & N-1 & N-1&N-1&(N-1)(N-2) \\
                                                  1&N-1 &-1&-1&2-N\\
                                                 1 &-1&N-1&-1&2-N \\
                                                 1 &-1&-1&N-1&2-N \\
                                                 1 &-1&-1&-1& 2
                         \end{pmatrix}.                                                                                  
\end{align*}
Squaring the above matrix, we obtain  $S^2=I$. Consequently, $S$ sends the basis \eqref{eq:5s} to the basis \eqref{eq:5}.
\noindent  Using our matrix representations we find that on $\Lambda$,
\begin{align*}
S A^{(r)} = A^{*(r)} S, \qquad \qquad SA^{*(r)} = A^{(r)} S \qquad \qquad r \in \lbrace 1,2,3\rbrace.
\end{align*}

\noindent This yields the following results. With respect to the basis \eqref{eq:5s}, the matrices representing  $A^{(1)}$, $A^{(2)}$, $A^{(3)}$ are
\begin{align*}
& A^{(1)}: \quad {\rm diag} \bigl( N-1, N-1, -1,-1,-1), \\
& A^{(2)}: \quad {\rm diag} \bigl( N-1, -1, N-1,-1,-1), \\
& A^{(3)}: \quad {\rm diag} \bigl( N-1, -1, -1,N-1,-1).
\end{align*}

\noindent With respect to the basis \eqref{eq:5s}, the matrices representing $A^{*(1)}$, $A^{*(2)}$, $A^{*(3)}$ are
\begin{align*}
&A^{*(1)}: \quad  \begin{pmatrix} 0 & N-1 & 0&0&0 \\
                                                  1&N-2 &0&0&0 \\
                                                 0 &0&0&1&N-2 \\
                                                 0 &0&1&0&N-2 \\
                                                 0 &0&1&1& N-3
                         \end{pmatrix}, \\
 &A^{*(2)}: \quad  \begin{pmatrix} 0 & 0 & N-1&0&0 \\
                                                  0&0 &0&1&N-2\\
                                                 1 &0&N-2&0&0 \\
                                                 0 &1&0&0&N-2 \\
                                                 0 &1&0&1& N-3
                         \end{pmatrix}, \\          
 &A^{*(3)}: \quad  \begin{pmatrix} 0 & 0 & 0&N-1&0 \\
                                                  0&0 &1&0&N-2\\
                                                 0 &1&0&0&N-2 \\
                                                 1 &0&0&N-2&0 \\
                                                 0 &1&1&0& N-3
                         \end{pmatrix}.                                                                                  
\end{align*}

\section{Directions for future research}
\noindent In this section, we give some conjectures and open problems.
\medskip

\begin{problem}\rm Recall the tridiagonal algebra $T=T(\beta, \gamma, \gamma^*, \varrho, \varrho^*)$ from Definition \ref{def:T}, and
the $S_3$-symmetric tridiagonal algebra $\mathbb T= \mathbb T(\beta, \gamma, \gamma^*, \varrho, \varrho^*)$ from Definition \ref{def:symT}. Let $W$ denote a
finite-dimensional irreducible $\mathbb T$-module.
 In Lemma \ref{lem:TtoSymT},  we used  distinct $r,s \in \lbrace 1,2,3\rbrace$  to obtain an  algebra homomorphism $T \to \mathbb T$. Pulling back the $\mathbb T$-action on $W$ via this homomorphism, we turn the $\mathbb T$-module $W$
into  a $T$-module. We can choose the ordered pair $r,s$ in six ways,  so $W$ becomes a $T$-module in six ways. Investigate how these six $T$-modules are related.
 Specifically, how are the irreducible $T$-submodules of $W$ with respect to one $T$-module structure, related to the irreducible $T$-submodules of $W$ with respect to another $T$-module structure?
\end{problem}

\noindent Next, we have two conjectures. These conjectures refer to the following situation. Consider the $Q$-polynomial distance-regular graph  $\Gamma=(X,\mathcal R)$ from Section 5.  
 Recall the scalars $\beta, \gamma, \gamma^*, \varrho, \varrho^*$ from Lemma \ref{lem:TTR}. Consider the
 standard module $V$ and the vector space $V^{\otimes 3}$.
 In Theorem \ref{thm:main} we turned  $V^{\otimes 3}$ into a module for the algebra $\mathbb T = \mathbb T(\beta, \gamma, \gamma^*, \varrho, \varrho^*)$.
 In Definition \ref{def:fun} we introduced the fundamental  $\mathbb T$-submodule $\Lambda$ of $V^{\otimes 3}$.
 By Lemmas \ref{lem:ABCs}, \ref{lem:PLam}, \ref{lem:Pzero} the following holds for $0 \leq h,i,j\leq D$:
 \begin{align*}
 E_h^{*(1)} E_i^{*(2)} E_j^{*(3)} \Lambda =0 \qquad {\mbox{\rm if and only if}} \qquad p^h_{i,j} =0.
\end{align*}
\begin{conjecture} \label{conj:qhij} For the above $\Lambda$, the following holds  for  $0 \leq h,i,j\leq D$:
  \begin{align*}
 E_h^{(1)} E_i^{(2)} E_j^{(3)} \Lambda =0 \qquad {\mbox{\rm if and only if}} \qquad q^h_{i,j} =0.
\end{align*}
\end{conjecture}
\begin{conjecture} \label{conj:PQ} For the above $\Lambda$, we list some subspaces along with a conjectured basis:
\begin{align*} 
&\begin{tabular}[t]{c|cccc}
{\rm Subspace}& $E_0^{*(1)}\Lambda$ & $E_0^{*(2)}\Lambda$ & $E_0^{*(3)}\Lambda$ 
 \\
 \hline
 {\rm Conjectured basis} & $\lbrace P_{0,i,i} \rbrace_{i=0}^D$ & $\lbrace P_{i,0,i} \rbrace_{i=0}^D$ & $\lbrace P_{i,i,0} \rbrace_{i=0}^D$
   \end{tabular}
\\
  &\begin{tabular}[t]{c|cccc}
{\rm Subspace}&  $E_0^{(1)}\Lambda$ & $E_0^{(2)}\Lambda$ &$E_0^{(3)}\Lambda$
 \\
 \hline
 {\rm Conjectured basis} & $\lbrace Q_{0,i,i} \rbrace_{i=0}^D$ & $\lbrace Q_{i,0,i} \rbrace_{i=0}^D$ & $\lbrace Q_{i,i,0} \rbrace_{i=0}^D$ 
   \end{tabular}
\end{align*}

\end{conjecture}

\begin{problem}\label{prob:translation}\rm  Investigate Conjectures \ref{conj:qhij}, \ref{conj:PQ} for the case in which an abelian subgroup $G$ of ${\rm Aut}(\Gamma)$ acts regularly on $X$; in this case $\Gamma$ is called a translation scheme
\cite{tanakaDual}.
\end{problem}

\noindent Before stating the next problem, we have some comments. 
For the moment, assume that $\Gamma$ is the Hamming graph $H(D,N)$ from Example \ref{ex:hypercube}. By Corollary \ref{cor:basisLam},  the fundamental $\mathbb T$-module $\Lambda$ has dimension $\binom{D+4}{4}$.
According to \cite[Corollary~3.5]{levstein}, for each vertex of $\Gamma$ the corresponding subconstituent algebra has dimension $\binom{D+4}{4}$. A vector space isomorphism from $\Lambda$ to this subconstituent
algebra, is given by \cite[Proposition~3]{tanakaCode} and Proposition \ref{prop:LamF}  above.
Based on these remarks, it is tempting to guess
 that for any vertex of any $Q$-polynomial
distance-regular graph, there is a vector space isomorphism from the corresponding subconstituent algebra to the fundamental $\mathbb T$-module $\Lambda$. However, a result about
the twisted Grassmann graph \cite[Theorem~6.2]{twist2} suggests that this isomorphism does not exist in general. 

\begin{problem}\rm Let $\Gamma$ denote a $Q$-polynomial distance-regular graph. Investigate the relationship between the subconstituent algebras of $\Gamma$
and the fundamental $\mathbb T$-module $\Lambda$. To illuminate this relationship, it might help to study the following examples: the Johnson graphs \cite[Section~9.1]{bcn}, the Grassmann graphs \cite[Section~9.3]{bcn}, and the dual polar graphs \cite[Section~9.4]{bcn}.
\end{problem}

\section{Acknowledgement} 
\noindent The author thanks Bill Martin for many helpful discussions.  The author thanks Kazumasa Nomura, for reading the manuscript carefully and sending comments.
The author thanks Hajime Tanaka, for  pointing out reference
 \cite{tanakaCode} in connection with Lemma  \ref{lem:orb} and reference \cite{tanakaDual} in connection with  Conjecture \ref{conj:qhij}.


\bigskip


\noindent Paul Terwilliger \hfil\break
\noindent Department of Mathematics \hfil\break
\noindent University of Wisconsin \hfil\break
\noindent 480 Lincoln Drive \hfil\break
\noindent Madison, WI 53706-1388 USA \hfil\break
\noindent email: {\tt terwilli@math.wisc.edu }\hfil\break

\section{Statements and Declarations}

\noindent {\bf Funding}: The author declares that no funds, grants, or other support were received during the preparation of this manuscript.
\medskip

\noindent  {\bf Competing interests}:  The author  has no relevant financial or non-financial interests to disclose.
\medskip

\noindent {\bf Data availability}: All data generated or analyzed during this study are included in this published article.

\end{document}